\documentclass[reqno]{amsart}

\usepackage[english]{babel}

\usepackage{amstext,amsmath,amsthm,amssymb,mathrsfs,bbm}
\usepackage{float}

\usepackage{dsfont}
\usepackage{bm}

\usepackage{txfonts}
\usepackage[T1]{fontenc}

\usepackage{amssymb,amsmath,amsfonts,amsthm,mathrsfs}
\usepackage[colorlinks=true, citecolor=blue, anchorcolor=red]{hyperref}

\usepackage{verbatim}
\usepackage{enumerate}
\usepackage{graphicx}
\usepackage{calc}
\usepackage[ansinew]{inputenc}

\usepackage{oldgerm}

\usepackage{float}

\usepackage{bm}

\usepackage{txfonts}
\usepackage[T1]{fontenc}

\usepackage{amstext}

\usepackage[colorinlistoftodos,prependcaption,textsize=tiny]{todonotes}

\usepackage{bm}

\usepackage{txfonts}
\DeclareMathAlphabet{\mathpzc}{OT1}{pzc}{m}{it}

\newcommand{\mc}[1]{\mathcal #1}

\newcommand{\R}{\mathbb R}
\newcommand{\N}{\mathbb N}
\newcommand{\C}{\mathbb C}

\newcommand{\E}{\mathbb E}

\newcommand{\op}[1]{\operatorname{#1}}
\newcommand{\ms}[1]{\mathscr{#1}}

\renewcommand{\epsilon}{\varepsilon}
\renewcommand{\bar}[1]{\overline{#1}}

\renewcommand{\tilde}{\widetilde}

\newtheorem{theorem}{Theorem}[section]
\newtheorem*{theorem*}{Theorem}
\newtheorem{lemma}[theorem]{Lemma}
\newtheorem{corollary}[theorem]{Corollary}
\newtheorem{proposition}[theorem]{Proposition}
\theoremstyle{definition}
\newtheorem{definition}[theorem]{Definition}
\newtheorem{remark}[theorem]{Remark}
\newtheorem{example}[theorem]{Example}

\newcommand{\dist}{{\rm dist}}

\numberwithin{equation}{section}

\renewcommand{\vec}[1]{\bm{#1}}

\makeatletter
\@namedef{subjclassname@2020}{\textup{2020} Mathematics Subject Classification}
\makeatother
\begin{document}

\title[]{Sample Paths of White Noise in Spaces with Dominating Mixed Smoothness}
\author{Felix Hummel}
\address{Technical University of Munich\\ Department of Mathematics \\ Boltzmannstra{\ss}e 3\\ 85748 Garching bei M\"unchen \\ Germany}
\email{hummel@ma.tum.de}
\thanks{}
\date{\today}

\begin{abstract}
    The sample paths of white noise are proved to be elements of certain Besov spaces with dominating mixed smoothness. Unlike in isotropic spaces, here the regularity does not get worse with increasing space dimension. Consequently, white noise is actually much smoother than the known sharp regularity results in isotropic spaces suggest. An application of our techniques yields new results for the regularity of solutions of Poisson and heat equation on the half space with boundary noise. The main novelty is the flexible treatment of the interplay between the singularity at the boundary and the smoothness in tangential, normal and time direction.
\end{abstract}
\maketitle

\section{Introduction}
There are many works studying the regularity of different kinds of stochastic noise. Oftentimes, regularity results are formulated in terms of Besov spaces. Classical results on the H\"older regularity of sample paths of a Brownian motion have been improved by using Besov spaces and Besov-Orlicz spaces in \cite{Ciesielski_1991,Ciesielski_1993}. Similar results have been obtained for Feller processes in \cite{Schilling_1997,Schilling_1998,Schilling_2000}, for a summary see \cite[Section 5.5]{Boettcher_Schilling_Wang_2013}, and for Brownian motions with values in Banach spaces in \cite{Hytonen_Veraar_2008}. Closely related to these works are characterizations of the Besov regularity of white noise. For a Gaussian white noise on the torus such characterizations are given in \cite{Veraar_2011}. L\'evy white noise on the torus was studied in \cite{Fageot_Unser_Ward_2017}. Global regularity results for Gaussian and L\'evy white noise are given in \cite{Aziznejad_Fageot_Unser_2018} and \cite{Fageot_Fallah_Unser_2017}.\\
Most of these works have in common that the regularity results are shown to be sharp up to possibly some minor improvements in some of the references. For an $n$-dimensional Gaussian white noise it is shown for example, that it has a smoothness of exactly or almost $-\frac{n}{2}$ but not more than $-\frac{n}{2}$, depending on the scale of isotropic function spaces. In particular, regularity seems to get worse with increasing dimension. The aim of this paper is to show that these results can be improved for Gaussian as well as L\'evy white noise if one works with spaces of dominating mixed smoothness. Roughly speaking, the following results states that an $n$-dimensional Gaussian white noise has local smoothness $-\frac{1}{2}-\epsilon$ separately in each direction, while previous results state that it has regularity $-\frac{n}{2}$ simultanously in all directions.
\begin{theorem}\label{Thm:Completely_Mixed}
 Let $1<p<\infty$ and $\epsilon,T>0$. Then the restriction of an $n$-dimensional Gaussian white noise on $\R^n$ to $[0,T]^n$ has a modification $\eta$ such that
 \[
  \mathbb{P}\left(\eta\in S^{(-\tfrac{1}{2}-\epsilon,\ldots,-\tfrac{1}{2}-\epsilon)}_{p,p}B([0,T]^n)\right)=1.
 \]
\end{theorem}\noindent
In this theorem $S^{(-\tfrac{1}{2}-\epsilon,\ldots,-\tfrac{1}{2}-\epsilon)}_{p,p}B([0,T]^n)$ denotes a Besov space with dominating mixed smoothness. It can be identified with the iterated Besov space
\[
 B^{-\tfrac{1}{2}-\epsilon}_{p,p}\bigg([0,T];B^{-\tfrac{1}{2}-\epsilon}_{p,p}\big([0,T];\ldots B^{-\tfrac{1}{2}-\epsilon}_{p,p}([0,T])\ldots\big)\bigg)
\]
and with the tensor product
\[
 B^{-\tfrac{1}{2}-\epsilon}_{p,p}([0,T])\otimes_{\alpha_p}\ldots \otimes_{\alpha_p}B^{-\tfrac{1}{2}-\epsilon}_{p,p}([0,T]),
\]
which is defined as the closure of the algebraic tensor product with respect to the so-called $p$-nuclear tensor norm. We will explain these identifications later in this paper.\\
If one component is viewed as time, then a white noise is also sometimes called space-time white noise. In this case, it can also be insightful the split space and time in the description of the smoothness. This way, we obtain that a Gaussian space-time white noise has smoothness $-\tfrac{1}{2}$ in time and $-\frac{n-1}{2}-\epsilon$ in (the $n-1$-dimensional) space. More precisely, we have the follwing result:
\begin{theorem}\label{Thm:Partially_Mixed}
 Let $1<p,\tilde{p}<\infty$ and $\epsilon>0$. Then an $n$-dimensional Gaussian white noise on $\R^n$ has a modification $\eta$ such that
 \[
  \mathbb{P}\left(\eta\in B^{-1/2}_{\tilde{p},\infty}\big([0,T];B^{-\frac{n-1}{2}-\epsilon}_{p,p}(\R^{n-1},\langle\cdot\rangle^{1-n-\epsilon})\big)\right)=1.
 \]
 Here, $\langle\xi\rangle^{1-n-\epsilon}:=(1+|\xi|^2)^{\frac{1-n-\epsilon}{2}}$ is a weight function. The intervall $[0,T]$ corresponds to the time direction, while $\R^{n-1}$ corresponds to the space direction.
\end{theorem}\noindent
Note that compared to Theorem \ref{Thm:Completely_Mixed} we can include growth bounds in space this time.  Theorem \ref{Thm:Partially_Mixed} can be useful if one studies parabolic partial differential equations driven by noise. We will illustrate this by deriving regularity results for the heat equation with Dirichlet and Neumann boundary noise. The main tool in previous works such as \cite{Alos_Bonaccorsi_2002, Brzezniak_et_al_2015, DaPrato_Zabczyk_1993, Schnaubelt_Veraar_2011} for analyzing solutions of equations with boundary noise were power weights. These weights measure the distance to the boundary and are well suited to describe the singularities of solutions at the boundary. Our approach however adds more flexibility to the description of these singularities, as it allows one to treat regularity in time, tangential and normal directions separately. It will also enable us to analyze the behavior of solutions at the boundary in spaces of higher regularity.\\

This paper is structured as follows:
\begin{itemize}
 \item In Section \ref{Section:Preliminaries} we introduce weighted Besov spaces with dominating mixed smoothness, L\'evy white noise and vector-valued L\'evy processes and cite the most important results we need throughout the paper. While most of the results are well-known, it seems like the description of the dual spaces of Besov spaces with dominating mixed smoothness on the domain $[0,T]^n$ given in Proposition \ref{Prop:DualBesovDomain} has not been available in the literature before.
 \item Section \ref{Section:Smoothness_Properties} is the main part of this paper. Therein, we derive regularity results for L\'evy white noise in spaces with dominating mixed smoothness.
 \item As an application of some of our results, we derive new regularity properties of the solutions of Poisson and heat equation with Dirichlet and Neumann boundary noise in Section \ref{Section:BoundaryNoise}.
\end{itemize}

\subsection{Notations and Assumptions}
We write $\N=\{1,2,3,\ldots\}$ for the natural numbers starting from $1$ and $\N_0=\{0,1,2,\ldots\}$ for the natural numbers starting from $0$. Throughout the paper we take $n\in\N$ and write
\[
 \R^n_+:=\{x=(x_1,\ldots,x_n)\in\R^n: x_n>0\}.
\]
If $n=1$ we also just write $\R_+:=\R^1_+$. Given a real number $x\in\R$, we write
\[
 x_+:=[x]_+:=\max\{0,x\}.
\]
 The Bessel potential will be denoted by
\[
    \langle x\rangle:=(1+|x|^2)^{1/2}\quad(x\in\R^n).
\]
Given a Banach space $E$ we will write $E'$ for its topological dual. By $\mathscr{D}(\R^n;E)$, $\mathscr{S}(\R^n;E)$ and $\mathscr{S}'(\R^n;E)$ we denote the spaces of $E$-valued test functions, $E$-valued Schwartz functions and $E$-valued tempered distributions, respectively. If $E\in\{\R,\C\}$ then we will omit it in the notation. On $\mathscr{S}(\R^n;E)$ we define the Fourier transform
\[
 (\mathscr{F}f)(\xi):=\frac{1}{(2\pi)^{n/2}}\int_{\R^n} e^{-ix\xi} f(x)\,dx\quad(f\in \mathscr{S}(\R^n;E)).
\]
As usual, we extend it to $\mathscr{S}'(\R^n;E)$ by $[\mathscr{F}u](f):=u(\mathscr{F}f)$ for $u\in \mathscr{S}'(\R^n;E)$ and $f\in \mathscr{S}(\R^n)$. Given two topological spaces $X,Y$, we write $X\hookrightarrow Y$ if there is a canonical continuous embedding. We write $X\stackrel{d}{\hookrightarrow} Y$ if the range of this embedding is dense in $Y$. If $E_0$ and $E_1$ are two locally convex spaces, then the spaces of continuous linear operators from $E_0$ to $E_1$ will be denoted by $\mathcal{B}(E_0,E_1)$. If $E_0=E_1$, then we also write $\mathcal{B}(E_0)$.\\
Throughout the paper, we will assume that $(\Omega,\mathcal{F},\mathbb{P})$ is a complete probability space.


\section{Preliminaries}\label{Section:Preliminaries}
\subsection{Weights} A weight $w$ on $\R^n$ is a function $w\colon\R^n\to[0,\infty]$ which takes values in $(0,\infty)$ almost everywhere with respect to the Lebesgue measure. There are several interesting classes of weights one can consider.
\begin{definition} Let $w\colon\R^n\to[0,\infty]$ be a weight.
 \begin{enumerate}[(a)]
  \item We say that $w$ is an admissible weight if $w\in C^{\infty}(\R^n;(0,\infty))$ with the following properties:
    \begin{enumerate}[(i)]
        \item For all $\alpha\in\N_0^n$ there is a constant $C_{\alpha}$ such that
            \begin{align}\label{Eq:AdmissibleWeight1}
             |D^{\alpha}w(x)|\leq C_{\alpha} w(x)\quad\text{for all }x\in\R^n.
            \end{align}
        \item There are two constants $C>0$ and $s\geq0$ such that
            \begin{align}\label{Eq:AdmissibleWeight2}
             0<w(x)\leq C w(y)\langle x-y\rangle^s\quad\text{for all }x,y\in\R^n.
            \end{align}
    \end{enumerate}
    We write $W(\R^n)$ for the set of all admissible weights on $\R^n$.
  \item Let $1<p<\infty$. Then $w$ is called $A_p$ weight if
    \[
     [w]_{A_p}=\sup_{Q\text{ cube in }\R^n}\left(\frac{1}{\operatorname{Leb}_n(Q)}\int_Q w(x)\,dx\right)\left(\frac{1}{\operatorname{Leb}_n(Q)}\int_Q w(x)^{-\frac{1}{p-1}}\,dx\right)^{p-1}<\infty.
    \]
    The set of all $A_p$ weights on $\R^n$ will be denoted by $A_p(\R^n)$. Moreover, we write $A_{\infty}(\R^n):=\bigcup_{1<p<\infty} A_p(\R^n)$. Such weights are also called Muckenhoupt weights.
  \item Let $1<p<\infty$. Then $w$ is called $A_p^{loc}$ weight if
    \[
     [w]_{A_p^{loc}}=\sup_{Q\text{ cube in }\R^n,\,\operatorname{Leb}_n(Q)\leq1}\left(\frac{1}{\operatorname{Leb}_n(Q)}\int_Q w(x)\,dx\right)\left(\frac{1}{\operatorname{Leb}_n(Q)}\int_Q w(x)^{-\frac{1}{p-1}}\,dx\right)^{p-1}<\infty.
    \]
    The set of all $A_p^{loc}$ weights on $\R^n$ will be denoted by $A_p^{loc}(\R^n)$. Moreover, we write $A_{\infty}^{loc}(\R^n):=\bigcup_{1<p<\infty} A_p^{loc}(\R^n)$. Such weights are also called local Muckenhoupt weights.
 \end{enumerate}
\end{definition}

\begin{remark}
	The class of local Muckenhoupt weights $A_{\infty}^{loc}(\R^n)$ was introduced in \cite{Rychkov_2001} with the aim of unifying Littlewood-Paley theories for function spaces with admissible weights and Muckenhoupt weights. Accordingly, we have that $W(\R^n)\cup A_{\infty}(\R^n)\subset A_{\infty}^{loc}(\R^n)$. 
\end{remark}

\begin{example}\label{Example:Weight}
 In this paper, we are mainly work with weights of the form
\[
 \langle\,\cdot\,\rangle^{\rho}\colon\R^n\to\R,\,\xi\mapsto(1+|\xi|^2)^{\rho/2}
\]
for some $\rho\in\R$. It will be important for us to which class of weights this function belongs for different choices of $\rho\in\R$.
\begin{enumerate}[(a)]
 \item For all $\rho\in\R$ we have that $\langle\,\cdot\,\rangle^{\rho}\in W(\R^n)$, i.e. $\langle\,\cdot\,\rangle^{\rho}$ is an admissible weight. This can either be computed directly or one can use the following abstract arguments which in turn are based on simple direct computations:\\
 For \eqref{Eq:AdmissibleWeight1} one can recall that $\langle\,\cdot\,\rangle^{\rho}$ is the standard example of a so-called H\"ormander symbol of order $\rho$, see for example \cite[Chapter 2, \S1, Example 2$°$]{KumanoGo_1981}. Thus, we even have 
 \[
  |D^{\alpha} \langle\xi\rangle^{\rho}|\leq C_{\alpha,\rho} \langle\xi\rangle^{\rho-|\alpha|}
 \]
which trivially implies \eqref{Eq:AdmissibleWeight1}. In \eqref{Eq:AdmissibleWeight2} one can take $C=2^{|\rho|}$ and $s=|\rho|$ by Peetre's inequality, see for example \cite[Proposition 3.3.31]{Ruzhansky_Turunen_2010}.
\item\label{Example:A_p} It holds that $\langle\,\cdot\,\rangle^{\rho}\in A_p(\R^n)$ if and only if $-n<\rho<(p-1)n$. Again, one can directly verify this for example by a similar computation as in \cite[Example 9.1.7]{Grafakos_2009}. We also refer to \cite[Example 1.3]{Haroske_2008} where this has been observed for the equivalent weight
\[
 w_{0,\rho}(\xi):=\begin{cases}
                   1&\text{if }|\xi|\leq1,\\
                   |\xi|^{\rho}&\text{if }|\xi|\geq1.
                  \end{cases}
\]
\item It follows directly from part \eqref{Example:A_p} that $\langle\,\cdot\,\rangle^{\rho}\in A_{\infty}(\R^n)$ if and only if $-n<\rho$.
\end{enumerate}
\end{example}

\begin{definition}\label{Def:LebesgueBochnerSpaces}
 Let $E$ be a Banach space, $w\colon\R^n\to [0,\infty]$ a weight and $1\leq p<\infty$. Then the weighted Lebesgue-Bochner space $L_p(\R^n,w;E)$ is defined as the space of all strongly measurable functions $f\colon\R^n\to E$ such that
 \[
  \| f \|_{L_p(\R^n,w;E)}:=\left(\int_{\R^n} \|f(x)\|_{E}^p w(x)\,dx\right)^{1/p}<\infty
 \]
with the usual modification for $p=\infty$. As usual, functions which coincide on sets of measure $0$ are considered as equal.
\end{definition}

\begin{remark}
 For this work it is important to note that there are different conventions in the literature concerning the definition of weighted Lebesgue-Bochner spaces. Oftentimes, the expression $\| f \|_{L_p(\R^n,w;E)}$ is defined by $\| wf \|_{L_p(\R^n;E)}$, whereas in our case it is defined by $\| w^{1/p}f \|_{L_p(\R^n;E)}$. Unfortunately, we will have to refer to some articles which use the one and to other articles which use the other convention. Thus, we will explicitly mention if a certain reference does not use the convention of Definition \ref{Def:LebesgueBochnerSpaces}.
\end{remark}

\subsection{Weighted Function Spaces with Dominating Mixed Smoothness} \hspace{0mm} As general references for the theory of spaces with dominating mixed smoothness we would like to mention \cite{Schmeisser_2007,Triebel_2019,Vybiral_2006}. These spaces are mainly used in approximation theory. They can also be used to study boundary value problems with rough boundary data, see \cite{Hummel_2020}. Our aim here is to derive sharper regularity results for the sample paths of white noise.\\
In this section, let $l\in\N$ and $\mathpzc{d}=(\mathpzc{d}_1,\ldots,\mathpzc{d}_l)\in \N^l$ with $\mathpzc{d}_1+\ldots +\mathpzc{d}_l=n$. We write $\R^n_{\mathpzc{d}}$ if we split $\R^n$ according to $\mathpzc{d}$, i.e.
    \[
     \R^n_{\mathpzc{d}}:=\R^{\mathpzc{d}_1}\times\ldots\times \R^{\mathpzc{d}_l}.
    \]
Moreover, if we have such a splitting then for $x\in\R^n_{\mathpzc{d}}$ we write $x=(x_{1,\mathpzc{d}},\ldots,x_{l,\mathpzc{d}})$ with $x_{j,\mathpzc{d}}\in\R^{\mathpzc{d}_j}$, $j=1,\ldots,l$.
    \begin{definition}
    \begin{enumerate}[(a)]
     \item Let $\varphi_0\in\mathscr{D}(\R^n)$ be a smooth function with compact support such that $0\leq \varphi_0\leq 1$,
        \[
         \varphi_0(\xi)=1\quad\text{if }|\xi|\leq 1,\qquad\varphi_0(\xi)=0\quad\text{if }|\xi|\geq 3/2.
        \]
    For $\xi\in\R^n$ and $k\in\N$ let further
    \begin{align*}
     \varphi(\xi)&:=\varphi_0(\xi)-\varphi_{0}(2\xi),\\
     \varphi_k(\xi)&:=\varphi(2^{-k}\xi).
    \end{align*}
    We call such a sequence $(\varphi_k)_{k\in\N_0}$ smooth dyadic resolution of unity and write $\Phi(\R^n)$ for the space of all such sequences.
    \item Let $E$ be a Banach space. To a smooth dyadic resolution of unity $(\varphi_k)_{k\in\N_0}\in\Phi(\R^n)$ we associate the sequence of operators $(S_k)_{k\in\N_0}$ on the space of tempered distributions $\mathscr{S}'(\R^n;E)$ by means of
    \[
     S_kf:=\mathscr{F}^{-1}\varphi_k\mathscr{F} f\quad(f\in\mathscr{S}'(\R^n;E)).
    \]
    The sequence $(S_kf)_{k\in\N_0}$ is called dyadic decomposition of $f$.
    \item For $j\in\{1,\ldots,l\}$ let $(\varphi^{(j)}_{k_j})_{k_j\in\N_0}\in \Phi(\R^{\mathpzc{d}_j})$ be a smooth dyadic resolution of unity on $\R^{\mathpzc{d}_j}$. Then we define
    \[
     \varphi_{\bar{k}}:=\bigotimes_{j=1}^l \varphi^{j}_{k_j},\quad S_{\bar{k}}=\mathscr{F}^{-1}\varphi_{\bar{k}}\mathscr{F}\quad(\bar{k}=(k_1,\ldots,k_l)\in\N_0^l).
    \]
    We write $\Phi(\R^n_{\mathpzc{d}})$ for all such $(\varphi_{\bar{k}})_{\bar{k}\in\N_0^l}$.
    \end{enumerate}
    \end{definition}
    
    \begin{definition}
     Let $w\colon\R^n\to [0,\infty]$ be a weight, $E$ a Banach space, $(\varphi_{\bar{k}})_{\bar{k}\in\N_0^l}\in \Phi(\R^n_{\mathpzc{d}})$, $\bar{s}=(s_1,\ldots,s_l)\in \R^l$ and $p,q\in[1,\infty]$.
     \begin{enumerate}[(a)]
      \item We define the Besov space with dominating mixed smoothness $S^{\bar{s}}_{p,q}B(\R^n_{\mathpzc{d}},w;E)$ as the space of all tempered distributions $f\in\mathscr{S}'(\R^n_{\mathpzc{d}};E)$ such that
     \[
      \|f\|_{S^{\bar{s}}_{p,q}B(\R^n_{\mathpzc{d}},w;E)}:=\bigg(\sum_{\bar{k}\in\N_0^l}2^{q\bar{s}\cdot\bar{k}}\|S_{\bar{k}} f\|^q_{L_p(\R^n_{\mathpzc{d}},w;E)}\bigg)^{1/q}<\infty
     \]
    with the usual modification for $q=\infty$. 
    \item The respective space on some domain $\mathcal{O}_{\mathpzc{d}}\subset\R^n_{\mathpzc{d}}$ is defined by restriction:
    \[
     S^{\bar{s}}_{p,q}B(\mathcal{O}_{\mathpzc{d}},w;E):=\{ f\vert_{\mathcal{O}_{\mathpzc{d}}}:f\in S^{\bar{s}}_{p,q}B(\R^n_{\mathpzc{d}},w;E)\}
    \]
    and
    \[
     \| f \|_{S^{\bar{s}}_{p,q}B(\mathcal{O}_{\mathpzc{d}},w;E)}:=\inf_{g\in S^{\bar{s}}_{p,q}B(\R^n_{\mathpzc{d}},w;E),\,g\vert_{\mathcal{O}_{\mathpzc{d}}}=f}\| g\|_{S^{\bar{s}}_{p,q}B(\R^n_{\mathpzc{d}},w;E)}.
    \]
    \item The space $S^{\bar{s}}_{p,q,0}B(\mathcal{O}_{\mathpzc{d}},w;E)$ is defined as the closure of the space $\mathscr{S}_0(\mathcal{O})$ of Schwartz functions with support in $\overline{\mathcal{O}_{\mathpzc{d}}}$ in the space $S^{\bar{s}}_{p,q}B(\R^n_{\mathpzc{d}},w;E)$.
     \end{enumerate}
    \end{definition}
    
    \begin{remark} \label{Remark:Besov_Dominating_Mixed}
     \begin{enumerate}[(a)]
      \item If $l=1$, then we obtain the usual definition of isotropic weighted vector-valued Besov spaces. In this case, following the usual convention we write $B^s_{p,q}$ and $B^s_{p,q,0}$ instead of $S^s_{p,q}B$ and $S^s_{p,q,0}B$, respectively.
      \item It is intentional that in the definition of $S^s_{p,q,0}B(\mathcal{O}_{\mathpzc{d}},w;E)$ we take the closure in the space $S^s_{p,q}B(\R^n_{\mathpzc{d}},w;E)$ and not in $S^s_{p,q}B(\mathcal{O}_{\mathpzc{d}},w;E)$. Even in the isotropic case there is a subtle difference between the two definitions for $s-\frac{1}{p}\in\N_0$ . We refer for example to \cite[Section 4.3.2]{Triebel_1978}, where this is carefully discussed for isotropic spaces. Therein, the spaces $\tilde{B}^s_{p,q}$ correspond to the definition with the closure in $B^s_{p,q}(\R^n,w;E)$, while $\mathring{B}^s_{p,q}$ corresponds to the definition with the closure in $B^s_{p,q}(\mathcal{O},w;E)$.
      \item \label{Remark:Besov_Dominating_Mixed_Tensor}There are special representations if $p=q<\infty$. For example, it was shown in \cite{Sickel_Ullrich_2009} that for $l=n$ we have the tensor product representation
      \[
       S^{\bar{s}}_{p,p}B(\R^n_{\mathpzc{d}})\cong B^{s_1}_{p,p}(\R)\otimes_{\alpha_p} S^{(s_2,\ldots,s_n)}_{p,p}B(\R^{n-1}_{(\mathpzc{d}_2,\ldots,\mathpzc{d}_n)}) \cong B^{s_1}_{p,p}(\R)\otimes_{\alpha_p}\ldots\otimes_{\alpha_p} B^{s_n}_{p,p}(\R),
      \]
    where the tensor product is the closure of the unique tensor product on tempered distributions in the sense of \cite[Lemma B.3]{Sickel_Ullrich_2009} with respect to the $p$-nuclear tensor norm $\alpha_p$, see \cite[Appendix B]{Sickel_Ullrich_2009}. For two Banach spaces $E_1, E_2$ the $p$-nuclear tensor norm is defined by
    \[
     \alpha_p(h,E_1,E_2):=\inf\left\{\left(\sum_{j=1}^N \| x_j\|_{E_1}^p\right)^{1/p}\cdot\sup\bigg\{\left(\sum_{j=1}^N|\lambda_j(y_j)|^{p'}\right)^{1/p'}:\lambda_j\in E_2', \|\lambda_j\|_{E_2'}=1\bigg\}\right\},
    \]
    where $p'$ denotes the conjugated H\"older index and where the infimum is taken over all representations $h=\sum_{j=1}^N x_j\otimes y_j$ for $N\in\N$, $x_{1},\ldots,x_N\in E_1$ and $y_{1},\ldots,y_N\in E_2$.
    \item In a certain parameter range one can also view a Besov space with dominating mixed smoothness as a Besov space with values in another Besov space. Since it seems like this has not been formulated in the literature so far, we make this more precise in the following.
     \end{enumerate}
    \end{remark}

\begin{theorem}\label{Thm:Kernel_Theorem}
 Let $E$ be a reflexive Banach space and $l=2$. Then there are unique isomorphisms
 \begin{align*}
    I_1\colon \mathscr{S}'(\R^n;E)\to \mathcal{B}(\mathscr{S}(\R^{\mathpzc{d}_1}_{x_{1,\mathpzc{d}}}),\mathscr{S}'(\R^{\mathpzc{d}_2}_{x_{2,\mathpzc{d}}};E)),\\
    I_2\colon \mathscr{S}'(\R^n;E)\to \mathcal{B}(\mathscr{S}(\R^{\mathpzc{d}_2}_{x_{2,\mathpzc{d}}}),\mathscr{S}'(\R^{\mathpzc{d}_1}_{x_{1,\mathpzc{d}}};E))
 \end{align*}
such that for all $u\in\mathscr{S'}(\R^n;E)$ and all $\varphi_1\in\mathscr{S}(\R^{\mathpzc{d}_1}_{x_{1,\mathpzc{d}}}),\varphi_2\in\mathscr{S}(\R^{\mathpzc{d}_1}_{x_{2,\mathpzc{d}}})$ it holds that
\[
 [[I_1(u)](\varphi_1)](\varphi_2)=u(\varphi_1\otimes \varphi_2)=[[I_2(u)](\varphi_2)](\varphi_1).
\]
\end{theorem}
\begin{proof}
    This is one of the kernel theorems from \cite[Appendix, Theorem 1.8.9]{Amann_2019}.
\end{proof}

\begin{proposition}\label{Prop:DominatingMixedSmoothnessFubiniBesov}
 Let $E$ be a Banach space, $\bar{s}=(s_1,s_2)\in\R^2$ and let $w_j\colon\R^{\mathpzc{d_j}}\to[0,\infty]$ $(j=1,2)$ be weights. Suppose that $w=w_1\otimes w_2$ and that $1< p<\infty$. The mappings $I_1,I_2$ from Theorem \ref{Thm:Kernel_Theorem} yield the following isomorphies:
\begin{align*}
 B^{s_1}_{pp}\big(\R^{\mathpzc{d}_1}_{x_{1,\mathpzc{d}}},w_1;B^{s_2}_{pp}(\R^{\mathpzc{d}_2}_{x_{2,\mathpzc{d}}},w_2;E)\big)\stackrel{I_1}{\cong}S^{\bar{s}}_{p,p}B(\R^n_{\mathpzc{d}},w,E)\stackrel{I_2}{\cong}B^{s_2}_{pp}\big(\R^{\mathpzc{d}_2}_{x_{2,\mathpzc{d}}},w_2;B^{s_1}_{pp}(\R^{\mathpzc{d}_1}_{x_{1,\mathpzc{d}}},w_1;E)\big).
\end{align*}
\end{proposition}
\begin{proof}
The assertion follows from Theorem \ref{Thm:Kernel_Theorem} and 
{\allowdisplaybreaks{
 \begin{align*}
  \|f\|_{S^{\bar{s}}_{p,p}B(\R^n_{\mathpzc{d}},w;E)}^p&=\sum_{\vec{k}\in\N_0^2}2^{p\bar{s}\cdot \vec{k}}\int_{\R^{\mathpzc{d_1}}}\int_{\R^{\mathpzc{d_2}}} \|S_{\vec{k}}f(x)\|_E^pw_2(x_{2,\mathpzc{d}})\,dx_{2,\mathpzc{d}}\,w_1(x_{1,\mathpzc{d}})\,dx_{1,\mathpzc{d}}\\
  &=\sum_{k_1\in\N_0}2^{ps_1k_1}\int_{\R^{\mathpzc{d_1}}}\sum_{k_2\in\N_0}2^{ps_2k_2}\int_{\R^{\mathpzc{d_2}}} \|S_{k_2}S_{k_1}f(x)\|_E^pw_2(x_{2,\mathpzc{d}})\,dx_{2,\mathpzc{d}}\,w_1(x_{2,\mathpzc{d}})\,dx_{1,\mathpzc{d}}\\
  &=\sum_{k_1\in\N_0}2^{ps_1k_1}\int_{\R^{\mathpzc{d_1}}}\|S_{k_1}f(x_{1,\mathpzc{d}},\,\cdot\,)\|^p_{B^{s_2}_{pp}(\R^{\mathpzc{d}_2},w_2;E)}w_1(x_{1,\mathpzc{d}})\,dx_{1,\mathpzc{d}}\\
  &=\|f\|_{B^{s_1}_{pp}(\R^{\mathpzc{d}_1}_{x_{1,\mathpzc{d}}},w_1;B^{s_2}_{pp}(\R^{\mathpzc{d}_2}_{x_{2,\mathpzc{d}}},w_2;E))}^p.
 \end{align*}}}
\end{proof}

\begin{remark}
 \begin{enumerate}[(a)]
  \item In Theorem \ref{Thm:Kernel_Theorem} and Proposition \ref{Prop:DominatingMixedSmoothnessFubiniBesov} we took $l=2$ only for notational convenience. The same arguments also work for $l\in\{3,\ldots,n\}$.
  \item In this work, we frequently use the representation in Proposition \ref{Prop:DominatingMixedSmoothnessFubiniBesov} of Besov spaces with dominating mixed smoothness. In the following, we omit the isomorphisms $I_1$ and $I_2$ in the notation and consider the spaces in Proposition \ref{Prop:DominatingMixedSmoothnessFubiniBesov} as equal.
 \end{enumerate}
\end{remark}

\begin{corollary}\label{Cor:Isomorphies_Dominating_Mixed_Domain}
 Let $T>0$, $l=n$, $\bar{s}=(s_1,\ldots,s_n)\in\R^n$ and $p\in[1,\infty)$. Then we have the isomorphisms
 \begin{align*}
  B^{s_1}_{p,p}([0,T];&B^{s_2}_{p,p}([0,T];\ldots B^{s_l}_{p,p}([0,T])\ldots))\cong S^{\bar{s}}_{p,p}B([0,T]^{n})\\
  &\cong B^{s_1}_{p,p}([0,T])\otimes_{\alpha_p}\ldots\otimes_{\alpha_p} B^{s_n}_{p,p}([0,T]).
 \end{align*}
\end{corollary}
\begin{proof}
 For $[0,T]$ being replaced by $\R$ these are the statements of Proposition~\ref{Prop:DominatingMixedSmoothnessFubiniBesov} and Remark~\ref{Remark:Besov_Dominating_Mixed}~ (\ref{Remark:Besov_Dominating_Mixed_Tensor}). Thus, the assertion follows by composing the isomorphisms with a suitable extension operator and the restriction to $[0,T]^n$.
\end{proof}

\begin{proposition}\label{Prop:DualBesov}
 Let $1<p,q<\infty$, $s\in\R$ and let $w\colon\R^n\to(0,\infty)$ be an admissible weight. Let further $p',q'\in(1,\infty)$ be the conjugated H\"older indices of $p$ and $q$, respectively. Then we have
 \[
  (B^s_{p,q}(\R^n,w))'=B^{-s}_{p',q'}(\R^n,w^{1-p'}).
 \]
\end{proposition}
\begin{proof}
 This result is taken from \cite[Chapter 5.1.2]{Schmeisser_Triebel_1987}. Note however that therein, a different convention concerning the notation of weighted spaces is used. The space $B^s_{p,q}(\R^n,w)$ in the notation of \cite{Schmeisser_Triebel_1987} corresponds to $B^s_{p,q}(\R^n,w^p)$ in our notation. Note also that the weights being considered in \cite{Schmeisser_Triebel_1987} are even much more general than the admissible weights we consider here.
\end{proof}

\begin{lemma}\label{Lemma:Iterated_Besov_Density}
 Let  $1<p<\infty$, $l=n$ and $\bar{s}=(s_1,\ldots,s_n)\in\R^n$. Then we have that $\mathscr{S}_0([0,T]^n)$ is dense in $B^{s_1}_{p,p,0}([0,T];B^{s_2}_{p,p,0}([0,T];\ldots B^{s_n}_{p,p,0}([0,T])\ldots))$. 
\end{lemma}
\begin{proof}
 Let $E$ be a Banach space. It holds that the algebraic tensor product $\mathscr{S}_0([0,T]^{n-1})\otimes B^{s_n}_{p,p,0}([0,T];E))$ is dense in $\mathscr{S}_0([0,T]^{n-1};B^{s_n}_{p,p,0}([0,T];E))$, see for example \cite[Theorem 1.3.6]{Amann_2003}. On the other hand, $\mathscr{S}_0([0,T];E)$ is by definition dense in $B^{s_n}_{p,p,0}([0,T];E)$.  Thus, we have the dense embeddings
 \begin{align*}
  \mathscr{S}_0([0,T]^{n-1})\otimes \mathscr{S}_0([0,T];E) \stackrel{d}{\hookrightarrow} \mathscr{S}_0([0,T]^{n-1})&\otimes B^{s_n}_{p,p,0}([0,T];E)) \\
  &\stackrel{d}{\hookrightarrow}\mathscr{S}_0([0,T]^{n-1};B^{s_n}_{p,p,0}([0,T];E)).
 \end{align*}
Since 
\[
 \mathscr{S}_0([0,T]^{n-1})\otimes \mathscr{S}_0([0,T];E)\subset \mathscr{S}_0([0,T]^n;E)\subset \mathscr{S}_0([0,T]^{n-1};B^{s_n}_{p,p,0}([0,T];E))
\]
we obtain that
\[
 \mathscr{S}_0([0,T]^n;E)\stackrel{d}{\hookrightarrow} \mathscr{S}_0([0,T]^{n-1};B^{s_n}_{p,p,0}([0,T];E)).
\]
Repeating the same argument for $\mathscr{S}_0([0,T]^{n-1};B^{s_n}_{p,p,0}([0,T];E))$ instead of $\mathscr{S}_0([0,T]^n;E)$ and iterating it, we obtain the assertion.
\end{proof}

\begin{corollary}\label{Cor:Isomorphies_Dominating_Mixed_Domain_0}
 Let  $1<p<\infty$, $l=n$ and $\bar{s}=(s_1,\ldots,s_n)\in\R^n$. Then we have that
 \[
  B^{s_1}_{p,p,0}([0,T];B^{s_2}_{p,p,0}([0,T];\ldots B^{s_n}_{p,p,0}([0,T])\ldots))\cong S^{\bar{s}}_{p,p,0}B([0,T]^{n})
 \]
where the isomorphism is the same as in Corollary \ref{Cor:Isomorphies_Dominating_Mixed_Domain}.
\end{corollary}
\begin{proof}
 By iteration we define $R_n:=\mathscr{S}_0([0,T])$ and
 \[
  R_{j-1}:=\{u\in \mathscr{S}_0([0,T];S^{(s_j,\ldots,s_n)}_{p,p}B([0,T]^{n+1-j}))\;\vert\; \forall t\in[0,T]:u(t)\in R_j\}\quad(j=2,\ldots,n).
 \]
Then we have $\mathscr{S}_0([0,T]^n)\subset R_1$ so that it follows together with Lemma \ref{Lemma:Iterated_Besov_Density} that
\begin{align*}
 B^{s_1}_{p,p,0}([0,T];\ldots B^{s_n}_{p,p,0}([0,T];E)&\ldots)\subset \overline{\mathscr{S}_0([0,T]^n)}\subset \overline{R_1}\\
 & \subset B^{s_1}_{p,p,0}([0,T];\ldots B^{s_n}_{p,p,0}([0,T];E)\ldots),
\end{align*}
where the closures are taken with respect to the topology of the iterated Besov space $ B^{s_1}_{p,p,0}([0,T];\ldots B^{s_n}_{p,p,0}([0,T];E)\ldots)$. Hence, we have that
\[
 B^{s_1}_{p,p,0}([0,T];\ldots B^{s_n}_{p,p,0}([0,T];E)\ldots)=\overline{\mathscr{S}_0([0,T]^n)}
\]
On the other hand, $S^{\bar{s}}_{p,p,0}B([0,T]^{n})$ is defined as the closure of $\mathscr{S}_0([0,T]^n)$ and thus, the assertion follows.
\end{proof}

\begin{proposition}\label{Prop:DualBesovDomain}
 Let $1<p<\infty$, $p'$ the conjugated H\"older index, $l=n$ and $\bar{s}=(s_1,\ldots,s_n)\in\R^n$. Then we have that
 \[
  (S^{\bar{s}}_{p,p}B([0,T]^{n}))'\cong S^{-\bar{s}}_{p',p',0}B([0,T]^{n}),\quad S^{\bar{s}}_{p,p,0}B([0,T]^{n}))'\cong S^{-\bar{s}}_{p',p'}B([0,T]^{n}).
 \]
\end{proposition}
\begin{proof}
It follows from Corollary \ref{Cor:Isomorphies_Dominating_Mixed_Domain} together with Corollary \ref{Cor:Isomorphies_Dominating_Mixed_Domain_0} that we can show the assertion on the level of iterated Besov spaces. Since they are defined by iteration, it suffices to show the assertion for the usual isotropic but vector-valued Besov-spaces on $[0,T]$, i.e. it suffices to show that
\[
 (B^{s}_{p,p}([0,T];E))'=B^{-s}_{p',p',0}([0,T];E'),\quad (B^{s}_{p,p,0}([0,T];E))'=B^{-s}_{p',p'}([0,T];E'),
\]
where $E$ is a reflexive Banach space. For these relations, we refer to \cite[Chapter VII, Theorem 2.8.4]{Amann_2019} or \cite[Theorem 11]{Muramatu_1973}. Even though the former reference considers different domains and the latter treats the scalar-valued situation, their extension-restriction methods also work in our setting.
\end{proof}

\begin{proposition}\label{Prop:Weighted_Besov_Isomorphism}
 Let $s,\rho\in\R$ and $1\leq p,q<\infty$. Then the mapping
 \[
  B^{s}_{p,q}(\R^n,\langle\,\cdot\,\rangle^{\rho})\to  B^{s}_{p,q}(\R^n),\;f\mapsto \langle\,\cdot\,\rangle^{\rho/p}f
 \]
 is an isomorphism
\end{proposition}
\begin{proof}
 Recall that $\langle\,\cdot\,\rangle^{\rho/p}$ is an admissible weight. This proposition actually holds for all admissible weights, see for example \cite[Theorem 6.5]{Triebel_2006}. Note that in this reference a different convention concerning the notation of weighted spaces is used.
\end{proof}

\begin{theorem}\label{Thm:Weighted_Besov_Complex_Interpolation}
	Let $p_0,p_1,q_0,q_1\in[1,\infty)$, $s_0,s_1\in\R$ and $w_0,w_1\in A_{\infty}^{loc}(\R^n)$. Let further $\theta\in(0,1)$ and 
	\[
		s=(1-\theta)s_0+\theta s_1,\quad\frac{1}{p}=\frac{1-\theta}{p_0}+\frac{\theta}{p_1},\quad \frac{1}{q}=\frac{1-\theta}{q_0}+\frac{\theta}{q_1},\quad w=w_0^{\frac{(1-\theta)p}{p_0}}w_1^{\frac{\theta p}{p_1}}.
	\]
	Then we have that
	\[
		[B^{s_0}_{p_0,q_0}(\R^n,w_0),B^{s_1}_{p_1,q_1}(\R^n,w_1)]_{\theta}=B^{s}_{p,q}(\R^n,w),
	\]
	where $[\cdot,\cdot]_{\theta}$ denotes the complex interpolation functor. In particular, it holds that
	\[
		B^{s_0}_{p_0,q_0}(\R^n,w_0) \cap B^{s_1}_{p_1,q_1}(\R^n,w_1) \subset B^{s}_{p,q}(\R^n,w)
	\]
\end{theorem}
\begin{proof}
	This is part of the statement of \cite[Theorem 4.5]{Sickel_et_al_2014}.
\end{proof}

In one proof, we also need Bessel potential spaces as a technical tool.
\begin{definition}
 Let $\bar{s}=(s_1,\ldots,s_l)\in \R^l$ and $p\in(1,\infty)$. Then we define $S^{\bar{s}}_pH(\R^n_{\mathpzc{d}})$ by
 \[
  S^{\bar{s}}_pH(\R^n_{\mathpzc{d}}):=\bigg\{f\in\mathscr{S}'(\R^n): \mathscr{F}^{-1}\prod_{j=1}^l\langle\xi_{j,\mathpzc{d}}\rangle^{s_j}\mathscr{F}f\in L_p(\R^n)\bigg\}
 \]
and endow it with the norm
\[
 \| f \|_{S^{\bar{s}}_pH(\R^n_{\mathpzc{d}})}:= \bigg\|\mathscr{F}^{-1}\prod_{j=1}^l\langle\xi_{j,\mathpzc{d}}\rangle^{s_j}\mathscr{F}f\bigg\|_{L_p(\R^n)}.
\]
If $l=1$ then we obtain the standard isotropic Bessel potential spaces and write $H^{s}_p(\R^n)$ instead.
\end{definition}

For Bessel potential spaces, we have to following embeddings: Let $\bar{s}=(s,\ldots,s)\in\R^n$ for some $s\in\R$. Then we have
\begin{align}\label{Eq:Bessel_potential_embedding1}
 H^{sn}_p(\R^n)\hookrightarrow S^{\bar{s}}_pH(\R^n)\hookrightarrow H^{s}_p(\R^n).
\end{align}
This can for example be found in \cite[(1.7)]{Schmeisser_2007} and \cite[(1.554)]{Triebel_2019}. If $p\in(1,2]$, then we have
\begin{align} \label{Eq:Bessel_potential_embedding2}
 B^{s}_{p,p}(\R^n)\hookrightarrow H^s_p(\R^n)
\end{align}
for which we refer to \cite[Theorem 6.4.4]{Bergh_Loefstroem_1976}. Moreover, for all $\epsilon>0$ we have that
\begin{align}\label{Eq:Bessel_potential_embedding3}
 H^s_p(\R^n)\hookrightarrow B^{s-\epsilon}_{p,p}(\R^n),
\end{align}
which can be obtained as a combination of \cite[Section 2.3.2, Proposition 2]{Triebel_1983} and \cite[Theorem 6.4.4]{Bergh_Loefstroem_1976}. As for Besov spaces, we have the tensor product representation
\begin{align}\label{Eq:Bessel_potential_tensor}
 S^{\bar{s}'}_pH(\R^{n-1}_{\mathpzc{d}'})\otimes_{\alpha_p}H^{s}_p(\R)\cong S^{\bar{s}}_pH(\R^{n}_{\mathpzc{d}}),
\end{align}
where $\bar{s}'=(s,\ldots,s)\in\R^{n-1}$, $\bar{s}=(s,\ldots,s)\in\R^{n}$, $\mathpzc{d}'=(1,\ldots,1)\in\N^{n-1}$ and $\mathpzc{d}=(1,\ldots,1)\in\N^{n}$. This has also been derived in \cite{Sickel_Ullrich_2009}.

\subsection{L\'evy White Noise}
Now we briefly introduce L\'evy white noise as a generalized random process and collect some of the known properties. In the following $\nu$ will be a L\'evy measure, i.e. a measure on $\R\setminus\{0\}$ such that $\int_{\R\setminus\{0\}} \min\{1, x^2\}\,d\nu(x)<\infty$. Moreover, we take $\gamma\in\R$ and $\sigma^2>0$. We call the triplet $(\gamma,\sigma^2,\nu)$ L\'evy triplet and the function
\[
 \Psi(\xi):= i\gamma\xi-\frac{\sigma^2\xi^2}{2}+\int_{\R\setminus\{0\}} (e^{ix\xi}-1-i\xi x\mathbbm{1}_{|x|\leq 1})\,d\nu(x)
\]
is called L\'evy exponent corresponding to the L\'evy triplet $(\gamma,\sigma^2,\nu)$. Functions of the form $\exp\circ\Psi$ for some L\'evy exponent $\Psi$ are exactly the characteristic functions of infinitely divisible random variables. \\
We endow the space of tempered distributions $\mathscr{S}'(\R^n)$ with the cylindrical $\sigma$-field $\mathcal{B}_c(\mathscr{S}'(\R^n))$ generated by the cylindrical sets, i.e. sets of the form
\[
 \{u\in\mathscr{S}'(\R^n): (\langle u,\varphi_1\rangle,\ldots, \langle u,\varphi_N\rangle)\in B\}
\]
for some $N\in\N$, $\varphi_1,\ldots,\varphi_N\in\mathscr{S}(\R^n)$ and some Borel set $B\in\mathcal{B}(\R^N)$. We will also consider $(\mathscr{S}'(\mathcal{O}),\mathcal{B}_c(\mathscr{S}'(\mathcal{O})))$ for certain domains $\mathcal{O}\subset \R^n$. We define this by restriction. More precisely, we write $\mathscr{S}_0(\mathcal{O})$ for the closed subspace of $\mathscr{S}(\R^n)$ which consists of functions with support in $\overline{\mathcal{O}}$. $\mathscr{S}'(\mathcal{O})$ is defined by
\[
 \mathscr{S}'(\mathcal{O}):=\{ u\vert_{\mathscr{S}_0(\mathcal{O})} : u\in \mathscr{S}'(\R^n) \},
\]
and $\mathcal{B}_c(\mathscr{S}'(\mathcal{O}))$ is the $\sigma$-field generated by sets of the form
\[
 \{u\in\mathscr{S}'(\mathcal{O}): (\langle u,\varphi_1\rangle,\ldots, \langle u,\varphi_N\rangle)\in B\}
\]
for some $N\in\N$, $\varphi_1,\ldots,\varphi_N\in\mathscr{S}_0(\mathcal{O})$ and some Borel set $B\in\mathcal{B}(\R^N)$. We also just write
\[
 u\vert_{\mathcal{O}}:=u\vert_{\mathscr{S}_0(\mathcal{O})}\quad u\in \mathscr{S}'(\R^n).
\]
This way, the mapping $u\mapsto u\vert_{\mathcal{O}}$ is a measurable mapping
\[
 (\mathscr{S}'(\R^n),\mathcal{B}_c(\mathscr{S}'(\R^n)))\to (\mathscr{S}'(\mathcal{O}),\mathcal{B}_c(\mathscr{S}'(\mathcal{O}))).
\]

\begin{definition}
 Let $(\Omega,\mathcal{F},\mathbb{P})$ be a probability space. A generalized random process $s$ is a measurable function
 \[
  s\colon (\Omega,\mathcal{F})\to (\mathscr{S}'(\R^n),\mathcal{B}_c(\mathscr{S}'(\R^n))).
 \]
The pushforward measure $\mathbb{P}_s$ defined by
\[
 \mathbb{P}_s(B):=\mathbb{P}(s^{-1}(B))\quad (B\in\mathcal{B}_c(\mathscr{S}'(\R^n)))
\]
is called probability law of $s$. Moreover, the characteristic functional $\widehat{\mathbb{P}}_s$ of $s$ is defined by
\[
 \widehat{\mathbb{P}}_s(\varphi):=\int_{\mathscr{S}'(\R^n)} \exp(i\langle u , \varphi \rangle)\,d\mathbb{P}_s(u).
\]
We will write $s(\omega)$ for the tempered distribution at $\omega\in\Omega$ and $\langle s, \varphi\rangle$ for the random variable which one obtains by testing $s$ against the Schwartz function $ \varphi\in\mathscr{S}(\R^n)$.
\end{definition}

In certain situations we also speak of a generalized random process if there only is a null set $N\subset\Omega$ such that the range of $s\vert_{\Omega\setminus N}$ is a subset of $\mathscr{S}'(\R^n)$. But since we assume our probability space to be complete, we may change every measurable mapping $f\colon(\Omega,\mathcal{F})\to (M,\mathcal{A})$ for a measurable space $(M,\mathcal{A})$ on arbitrary null sets without affecting the measurability. Thus, for our purposes we can neglect the difference between a generalized random process and a mapping which is a generalized random process only after some change on a null set. This also applies to the following definition:
\begin{definition}
 Let
 \[
  s_1,s_2\colon(\Omega,\mathcal{F})\to (\mathscr{S}'(\R^n),\mathcal{B}_c(\mathscr{S}'(\R^n)))
 \]
 be two generalized random processes. We say that $s_2$ is a modification of $s_1$, if 
\[
 \mathbb{P}(\langle s_1,\varphi\rangle=\langle s_2,\varphi\rangle)=1
\]
for all $\varphi\in\mathscr{S}(\R^n)$.
\end{definition}

Similar to Bochner's theorem for random variables, the Bochner-Minlos theorem gives a necessary and sufficient condition for a mapping $C\colon\mathscr{S}(\R^n)\to \C$ to be the characteristic functional of a generalized random process.
\begin{theorem}[Bochner-Minlos]
 A mapping $C\colon\mathscr{S}(\R^n)\to \C$ is the characteristic functional of a generalized random process if and only if $C$ is continuous, $C(0)=1$ and $C$ is positive definite, i.e. for all $N\in\N$, all $z_1,\ldots, z_N\in\C$, and all $\varphi_1,\ldots,\varphi_N\in\mathscr{S}(\R^n)$ it holds that
 \[
  \sum_{j,k=1}^N z_j\overline{z_k}C(\varphi_j-\varphi_k)\geq0.
 \]
\end{theorem}

\begin{remark}\vspace{0mm}\label{Rem:CharFunctional}
 \begin{enumerate}[(a)]
  \item The Bochner-Minlos also holds if $\mathscr{S}(\R^n)$ is replaced by a nuclear space as for example the space of test functions $\mathscr{D}(\R^n)$. It seems like the Bochner-Minlos theorem was first formulated and proved in \cite{Minlos_1959}.
  \item An important example of a characteristic functional is given by
  \[
   C(\varphi):=\exp\left(\int_{\R^n}\Psi(\varphi(x))\,dx \right)
  \]
    for a L\'evy exponent $\Psi$. This is always a characteristic functional on the space of test functions $\mathscr{D}(\R^n)$, see for example \cite[Chapter III, Theorem 5]{Gelfand_Vilenkin_1964}. However, this is not always true for the Schwartz space $\mathscr{S}(\R^n)$. In fact, $C$ is a characteristic functional on $\mathscr{S}(\R^n)$ if and only if it has positive absolute moments, i.e. if there is an $\epsilon>0$ such that $\E[|X|^{\epsilon}]<\infty$, where $X$ is an infinitely divisible random variable corresponding to the L\'evy triplet $\Psi$. We refer the reader to \cite[Theorem 3]{Fageot_Amini_Unser_2014} for the sufficiency and to \cite{Dalang_Humeau_2017} for the necessity.
 \end{enumerate}
\end{remark}

\begin{definition}\label{Def:WhiteNoise}
 Let $(\gamma,\sigma^2,\nu)$ be a L\'evy triplet such that the corresponding infinitely divisible random variable has positive absolute moments. A L\'evy white noise $\eta\colon \Omega\to\mathscr{S}'(\R^n)$ with L\'evy triplet $(\gamma,\sigma^2,\nu)$ is the generalized random process with characteristic functional
 \[
  \mathbb{P}_{\eta}(\varphi)=\exp\left(\int_{\R^n} \Psi(\varphi(x))\,dx\right)\quad(\varphi\in\mathscr{S}(\R^n)).
 \]
 If we speak of a L\'evy white noise on a domain $\mathcal{O}\subset\R^n$, then we mean that it is given by $\eta\vert_{\mathcal{O}}$ for a Levy white noise $\eta$ on $\R^n$.
\end{definition}

\begin{remark}\label{Rem:WhiteNoiseModel}
 From a modeling point of view, there are some minimum requirements one has on a random process to call it a white noise. For example, a white noise should and indeed our white noise from Definition \ref{Def:WhiteNoise} does satisfy the following:
  \begin{enumerate}[(a)]
  \item A white noise is invariant under Euclidean motions in the sense that for $f\in\mathscr{D}(\R^n)$ and for an Euclidean motion $A$ the random variables $\eta(f)$ and $\eta(f\circ A)$ have the same distribution. This can for example be seen by comparing their characteristic functions:
  \[
   \E[e^{i\xi\eta(f)}]=\exp\bigg(\int_{\R^n}\Psi(\xi f(x))\,dx\bigg)=\exp\bigg(\int_{\R^n}\Psi(\xi f(Ax))\,dx\bigg)=\E[e^{i\xi\eta(f\circ A)}]\quad(\xi\in\R).
  \]
    For the representation of the characteristic function, see for example \cite[Theorem 2.7 (iv)]{Rajput_Rosinski_1989}.
  \item The random variables $\eta(f)$ and $\eta(g)$ are independent if $f,g\in\mathscr{D}(\R^n)$ have disjoint supports. Indeed, if $f,g$ have disjoint supports then $\Psi(f+g)=\Psi(f) +\Psi(g)$ and therefore
  \begin{align*}
    &\E[e^{i(\xi_1\eta(f)+\xi_2\eta(g))}]=\exp\bigg(\int_{\R^n}\Psi(\xi_1 f(x)+\xi_2g(x))\,dx\bigg)\\
    &=\exp\bigg(\int_{\R^n}\Psi(\xi_1 f(x))+\Psi(\xi_2g(x))\,dx\bigg)
    =\E[e^{i\xi_1\eta(f)}]\E[e^{i\xi_2\eta(g)}]
  \end{align*}
  \item If second moments exist, then we have the relation
  \[
   \operatorname{cov}(\eta(f),\eta(g))=\langle f,g \rangle_{L_2(\R^n)}\quad(f,g\in\mathscr{D}(\R^n)).
  \]
    It seems like this has not been stated in this form for L\'evy white noise in the literature before. We therefore refer the reader to the author's Ph.D. thesis, \cite[Proposition 3.33]{Hummel_2019}.
 \end{enumerate}
\end{remark}

\begin{remark}\label{Rem:Extension}
 By an approximation procedure it is possible to plug many more functions into a white noise than just test functions or Schwartz functions. For example, it is always possible to apply a L\'evy white noise to elements of $L_2(\R^n)$ with compact support. In particular, this includes indicator functions $\mathbbm{1}_A$ for bounded Borel sets $A\in\mathcal{B}(\R^n)$ which is useful for the construction of a stochastic integral. The idea for the construction of such an integral goes back to \cite{Urbanik_Woyczynski_1967} and was further refined in \cite{Rajput_Rosinski_1989}. We also refer the reader to \cite{Fageot_Humeau_2017} in which the extension of the domain of definition is carried out in full detail. We will now briefly summarize the results we need in this work.
\end{remark}

\begin{definition}
 Let $\eta$ be a L\'evy white noise with triplet $(\gamma,\sigma^2,\nu)$ and let $p\geq0$.
 \begin{enumerate}[(a)]
  \item The $p$-th order Rajput-Risi\'nski exponent $\Psi_p$ of $\eta$ is defined by
  \[
   \Psi_p(\xi):=\bigg|\gamma\xi+\int_{\R\setminus\{0\}}x\xi(\mathbbm{1}_{|x\xi|\leq1}-\mathbbm{1}_{|x|\leq 1})\,d\nu(x)\bigg|+\sigma^2\xi^2+\int_{\R\setminus\{0\}} |x\xi|^p\mathbbm{1}_{|x\xi|>1}+|x\xi|^2\mathbbm{1}_{|x\xi|\leq1}\,d\nu(x)
  \]
  for $\xi\in\R$.
 \item We define the space $L_p(\eta)$ by
    \[
     L_p(\eta):=\bigg\{f\in L_0(\R^n):\int_{\R^n}\Psi_p(f(x))\,dx<\infty\bigg\}
    \]
    and endow it with the metric
    \[
     d_{\Psi_p}(f,g):=\inf\bigg\{\lambda>0:\int_{\R^n}\Psi_p\big(\tfrac{f(x)-g(x)}{\lambda}\big)<\lambda\bigg\}.
    \]
    The elements of $L_0(\eta)$ will be called $\eta$-integrable.
 \end{enumerate}
\end{definition}

\begin{proposition}\label{Prop:IntegrableFunctions} Let $\eta$ be a L\'evy white noise with triplet $(\gamma,\sigma^2,\nu)$ and $p\geq0$.
 \begin{enumerate}[(a)]
  \item The space $L_p(\eta)$ is a complete linear metric space.
  \item The space of test functions $\mathscr{D}(\R^n)$ is dense in $L_0(\eta)$.
  \item The L\'evy white noise $\eta$ extends to a continuous linear mapping
  \[
   \eta\colon L_p(\eta)\to L_p(\Omega),\,f\mapsto\langle\eta,f\rangle.
  \]
  \item Let $f\in L_0(\eta)$. Then the characteristic function of $\langle\eta,f\rangle$ is again given by
  \[
   \E[e^{i\xi\langle\eta,f\rangle}]=\exp\bigg(\int_{\R^n}\Psi(\xi f(x))\,dx\bigg).
  \]

 \end{enumerate}
\end{proposition}
\begin{proof}
 This is a collection of the statements given in \cite[Proposition 3.9]{Fageot_Humeau_2017}, \cite[Chapter X, Theorem 2, Proposition 5 \& Corollary 6]{Rao_Ren_1991} and \cite[Theorem 2.7, Lemma 3.1 \& Theorem 3.3]{Rajput_Rosinski_1989}.
\end{proof}

\begin{remark}\vspace{0mm}\label{Rem:Remark_Integrable_Functions}
 \begin{enumerate}[(a)]
 \item In the general case, it can be difficult to give a nice characterization of the space $L_0(\eta)$. However, as already mentioned in Remark \ref{Rem:Extension}, elements of $L_2(\R^n)$ with compact support are always contained in $L_0(\eta)$.  Moreover, $\mathscr{S}(\R^n)$ is contained in $L_0(\eta)$ if the white noise $\eta$ admits positive absolute moments, see Remark \ref{Rem:CharFunctional}. We also refer the reader to \cite[Table 1]{Fageot_Humeau_2017} which contains a list of examples. For instance, in the Gaussian case we have $L_0(\eta)=L_2(\R^n)$. The same holds if the L\'evy triplet is given by $(0,\sigma^2,\nu)$ with $\nu$ being symmetric and having finite variance, see \cite[Proposition 5.10]{Fageot_Humeau_2017}. If the L\'evy triplet is given by $(\gamma,0,0)$ ($\gamma\neq0$) then we have $L_0(\eta)=L_1(\R^n)$ and for $(\gamma,\sigma^2,0)$ ($\gamma\neq0$, $\sigma^2>0$) by $L_1(\R^n)\cap L_2(\R^n)$.
  \item If one wants to work with paths of a L\'evy white noise, then a characterization of the Besov regularity of these paths might be more useful than Proposition \ref{Prop:IntegrableFunctions} in certain situations. Fortunately, a lot of nice work has already been done in this direction. For example, local regularity of Gaussian white noise has been studied in \cite{Veraar_2011}. In \cite{Fageot_Unser_Ward_2017} similar results have been obtained for L\'evy white noise. Global smoothness properties of L\'evy white noise in weighted spaces have been established in \cite{Aziznejad_Fageot_Unser_2018} and \cite{Fageot_Fallah_Unser_2017}. Results as in the latter two references will be important for the derivation of mixed smoothness properties. But before we can formulate them, we first need to introduce the Blumenthal-Getoor indices and the moment index of a L\'evy white noise.
  \end{enumerate}
\end{remark}

\begin{definition}
 Let $\eta$ be a L\'evy white noise with L\'evy exponent $\Psi$. Then the Blumenthal-Getoor indices are defined by
 \begin{align*}
  \beta_{\infty}&:=\inf\bigg\{p>0:\lim_{|\xi|\to\infty}\frac{|\Psi(\xi)|}{|\xi|^p}=0\bigg\},\\
  \underline{\beta}{}_{\infty}&:=\inf\bigg\{p>0:\liminf_{|\xi|\to\infty}\frac{|\Psi(\xi)|}{|\xi|^p}=0\bigg\}.
 \end{align*}
In addition, the moment index is defined by
\[
 p_{\max}:=\sup\{p>0:\E[|\eta(\mathbbm{1}_{[0,1]^n})|^p]<\infty\}.
\]
\end{definition}

In general it holds that $0\leq \underline{\beta}{}_{\infty}\leq \beta_{\infty} \leq 2$.

\begin{theorem}\label{Thm:White_Noise_Regularity}
 Let $\eta$ be a L\'evy white noise with L\'evy triplet $(\gamma,\sigma^2,\nu)$, Blumenthal-Getoor indices $\beta_{\infty}, \underline{\beta}{}_{\infty}$ and moment index $p_{\max}$. Let further $p\in(0,\infty)$.
 \begin{enumerate}[(a)]
  \item Gaussian case:\\
  Suppose that $\nu=0$. Then it holds that
    \begin{align*}
     \mathbb{P}(\eta\in B^{s}_{p,p}(\R^n,\langle\cdot\rangle^{\rho}))&=1,\quad\text{if }s<\tfrac{n}{2} \text{ and }\rho<-n,\\
     \mathbb{P}(\eta\notin B^{s}_{p,p}(\R^n,\langle\cdot\rangle^{\rho}))&=1,\quad\text{if }s\geq\tfrac{n}{2}\text{ or }\rho\geq-n.
    \end{align*}
  \item Compound Poisson case:\\
  Suppose that $\nu$ is a finite measure on $\mathcal{B}(\R^n\setminus\{0\})$ and that $\sigma^2=0$. Then it holds that
    \begin{align*}
     \mathbb{P}(\eta\in B^{s}_{p,p}(\R^n,\langle\cdot\rangle^{\rho}))&=1,\quad\text{if }s<n(\tfrac{1}{p}-1) \text{ and }\rho<-\tfrac{np}{\min\{p,p_{\max}\}},\\
     \mathbb{P}(\eta\notin B^{s}_{p,p}(\R^n,\langle\cdot\rangle^{\rho}))&=1,\quad\text{if }s\geq n(\tfrac{1}{p}-1)\text{ or }\rho>-\tfrac{np}{\min\{p,p_{\max}\}}.
    \end{align*}
  \item \label{Thm:White_Noise_Regularity:General} General non-Gaussian case:\\
  Suppose that $\nu\neq0$ and and that $p\leq2$ or $p\in2\N$. Then it holds that
      \begin{align*}
     \mathbb{P}(\eta\in B^{s}_{p,p}(\R^n,\langle\cdot\rangle^{\rho}))&=1,\quad\text{if }s<n(\tfrac{1}{\max\{p,\beta_{\infty}\}}-1)\text{ and }\rho<-\tfrac{np}{\min\{p,p_{\max}\}},\\
     \mathbb{P}(\eta\notin B^{s}_{p,p}(\R^n,\langle\cdot\rangle^{\rho}))&=1,\quad\text{if }s>n(\tfrac{1}{\max\{p,\underline{\beta}{}_{\infty}\}}-1)\text{ or }\rho>-\tfrac{np}{\min\{p,p_{\max}\}}.
    \end{align*}
 \end{enumerate}
\end{theorem}
\begin{proof}
 This is a collection of Proposition 6, 9 and 12 from \cite{Aziznejad_Fageot_Unser_2018}. Note that the authors of \cite{Aziznejad_Fageot_Unser_2018} use a different convention concerning the notation of weighted Besov spaces so that the weight parameters in our formulation are multiplied by $p$ compared to the formulation in \cite{Aziznejad_Fageot_Unser_2018}.
\end{proof}

\begin{remark}\phantomsection \label{Rem:Dropping_Some_Conditions}
	\begin{enumerate}[(a)]
		\item In Theorem \ref{Thm:White_Noise_Regularity} \eqref{Thm:White_Noise_Regularity:General} one can weaken the restriction on $p$ by using Theorem \ref{Thm:Weighted_Besov_Complex_Interpolation} as follows: If there is $N\in2\N$ such that $p_{\max}\in(N,N+2)$ and if $p\in (1,\infty)\setminus(N,N+2)$, then it holds that
		\[
		     \mathbb{P}(\eta\in B^{s}_{p,p}(\R^n,\langle\cdot\rangle^{\rho}))=1,\quad\text{if }s<\tfrac{n}{\max\{p,\beta_{\infty}\}}-n\text{ and }\rho<-\tfrac{np}{\min\{p,p_{\max}\}}.
		\]
		If $p\in(N,N+2)$ then let $\theta\in(0,1)$ such that $1/p=(1-\theta)/N+\theta/(N+2)$. In this case, it holds that 
		\[
		     \mathbb{P}(\eta\in B^{s}_{p,p}(\R^n,\langle\cdot\rangle^{\rho}))=1,\quad\text{if }s<\tfrac{n}{\max\{p,\beta_{\infty}\}}-n\text{ and }\rho<-n(\tfrac{(1-\theta)p_{\max}+\theta p}{p_{\max}}).
		\]
		If there is no such $N$, i.e. if $p_{\max}\in2\N$, then
		\[
		     \mathbb{P}(\eta\in B^{s}_{p,p}(\R^n,\langle\cdot\rangle^{\rho}))=1,\quad\text{if }s<\tfrac{n}{\max\{p,\beta_{\infty}\}}-n\text{ and }\rho<-\tfrac{np}{\min\{p,p_{\max}\}}.
		\]
		without restriction on $p$.
		\item If one restricts the white noise $\eta$ to a bounded set, for example $[0,T]^n$ for some $T>0$, then one can also drop the conditions on $\rho$. More precisely, we have the following: In the Gaussian case it holds that
		  \begin{align*}
     \mathbb{P}(\eta\in B^{s}_{p,p}([0,T]^n))&=1,\quad\text{if }s<\tfrac{n}{2},\\
     \mathbb{P}(\eta\notin B^{s}_{p,p}([0,T]^n))&=1,\quad\text{if }s\geq\tfrac{n}{2}.
    \end{align*}
    In the compound Poisson case it holds that 
    		  \begin{align*}
     \mathbb{P}(\eta\in B^{s}_{p,p}([0,T]^n))&=1,\quad\text{if }s<n(\tfrac{1}{p}-1),\\
     \mathbb{P}(\eta\notin B^{s}_{p,p}([0,T]^n))&=1,\quad\text{if }s\geq n(\tfrac{1}{p}-1).
    \end{align*}
    In the general con-Gaussian case with $p\in(1,\infty)$ it holds that
        		  \begin{align*}
     \mathbb{P}(\eta\in B^{s}_{p,p}([0,T]^n))&=1,\quad\text{if }s<n(\tfrac{1}{\max\{p,\beta_{\infty}\}}-1),\\
     \mathbb{P}(\eta\notin B^{s}_{p,p}([0,T]^n))&=1,\quad\text{if }s\geq n(\tfrac{1}{\max\{p,\underline{\beta}{}_{\infty}\}}-1).
    \end{align*}
	\end{enumerate}
\end{remark}

\subsection{L\'evy processes with values in a Banach space}
We briefly derive some results on the regularity of sample paths of L\'evy processes with values in Banach spaces. While they are most probably far from being optimal, they allow us to also apply our methods to L\'evy white noise instead of just Gaussian white noise. Although our regularity results for L\'evy white noises will not be sharp, we develop our methods in a way such that the result can directly be improved once properties like the ones in \cite[Section 5.5]{Boettcher_Schilling_Wang_2013} have been derived for L\'evy processes in Banach spaces.

\begin{definition}
 Let $T>0$. As in the scalar-valued case, a stochastic process $(L_t)_{t\in[0,T]}$ with values in a Banach space $E$ is called L\'evy process if the following holds:
 \begin{enumerate}[(i)]
  \item $L_0=0$,
  \item $(L_t)_{t\in[0,T]}$ has independent increments, i.e. for all $N\in\N$ and all $0\leq t_1<\ldots<t_N\leq T$ it holds that $L_{t_2}-L_{t_1},\ldots L_{t_N}-L_{t_{N-1}}$ are independent.
  \item $(L_t)_{t\in[0,T]}$ has stationary increments, i.e. the law of $L_t-L_s$ only depends on $t-s$.
  \item $(L_t)_{t\in[0,T]}$ is continuous in probability.
 \end{enumerate}
\end{definition}

\begin{proposition}\label{Prop:Levy_Paths}
 Let $\epsilon>0$, $p\in(1,\infty)$ and let $(L_t)_{t\in[0,T]}$ a L\'evy process with values in a Banach space $E$. Then $(L_t)_{t\in[0,T]}$ has a modification with sample paths in $B^{0}_{p,p}([0,T];E)$ if $p\geq2$ and in $B^{-\epsilon}_{p,p}([0,T];E)$ if $p<2$..
\end{proposition}
\begin{proof}
 As a L\'evy process, $(L_t)_{t\in[0,T]}$ has a modification such that the sample paths are c\`adl\`ag, see \cite[Theorem 4.3]{Peszat_Zabczyk_2007}. In particular, the sample paths are jump continuous, i.e. they are contained in the closure of simple functions from $[0,T]$ to $E$ with respect to the $\|\cdot\|_{L_{\infty}([0,T])}$-norm. Therefore, the sample paths are elements of 
 \[
  L_{\infty}([0,T])\hookrightarrow L_p([0,T])\hookrightarrow B^0_{p,p}([0,T])
 \]
where the latter embedding only holds for $p\geq2$ and can for example be found in \cite[Theorem 6.4.4]{Bergh_Loefstroem_1976}. If $p\in(1,2)$, then the embedding $L_p([0,T])\hookrightarrow B^{-\epsilon}_{p,p}([0,T])$ holds.
\end{proof}

Proposition \ref{Prop:Levy_Paths} is surely not sharp, but simple and good enough for our purposes. Nonetheless, there are already much sharper results for $E$-valued Brownian motions.

\begin{theorem}\label{Thm:Brownian_Paths}
 Let $p,q\in[1,\infty)$ and let $(W_t)_{t\in[0,T]}$ be a Brownian motion with values in the Banach space $E$, i.e. $(W_t)_{t\in[0,T]}$ is an $E$-valued L\'evy process such that $(\lambda(W_t))_{t\geq0}$ is a Brownian motion for all $\lambda\in E'$. Then the sample paths of $(W_t)_{t\in[0,T]}$ are contained in $B^{\frac{1}{2}}_{p,\infty}([0,T];E)$ almost surely. Moreover, almost surely they are not contained in $B^{\frac{1}{2}}_{p,q}([0,T];E)$.
\end{theorem}
\begin{proof}
 This is one of the statements of \cite[Theorem 4.1]{Hytonen_Veraar_2008}.
\end{proof}


\section{Regularity Properties in Spaces of Mixed Smoothness}\label{Section:Smoothness_Properties}
\begin{lemma}\label{lemma:ProductIndicator_old}
 Let $n=n_1+n_2$ with $n_1,n_2\in\N$ and let $s,t\in\R^{n_1}$, $s\leq t$. Let $\eta_n$ be a L\'evy white noise in $\R^n$ with L\'evy triplet $(\gamma,\sigma^2,\nu)$ and let $\eta_{n_2}$ be a L\'evy white noise in $\R^{n_2}$ with the same L\'evy triplet. Then the mapping
 \[
  L_0(\eta_{n_2})\to L_0(\eta_{n}),\,\varphi\mapsto \mathbbm{1}_{(s,t]}\otimes \varphi
 \]
is well-defined and continuous.
\end{lemma}
\begin{proof}
 Note that for all $\lambda>0$ we have
 \[
  \int_{\R^{n_1}\times\R^{n_2}}\Psi_0\bigg(\frac{\mathbbm{1}_{(s,t]}(r_1) \varphi(r_2)}{\lambda}\bigg)\,d(r_1,r_2)=\operatorname{Leb}_{n_1}((s,t]) \int_{\R^{n_2}}\Psi_0\bigg(\frac{\varphi(r_2)}{\lambda}\bigg)\,dr_2\;\;(\varphi\in L(\eta_{n_2}),\lambda>0).
 \]
Thus, $ \varphi\in L_0(\eta_{n_2})$ implies $\mathbbm{1}_{(s,t]}\otimes \varphi\in L_0(\eta_{n})$. Moreover, if $(\varphi_k)_{k\in\N}\subset L_0(\eta_{n_2})$ converges to $\varphi$, then for all $\epsilon>0$ we have
\[
 \int_{\R^{n_2}}\Psi_0\bigg(\frac{\varphi(r_2)-\varphi_k(r_2)}{\epsilon}\bigg)\,dr_2<\epsilon
\]
for $k\in\N$ large enough. If $\operatorname{Leb}_{n_1}((s,t])\leq 1$ this implies
\[
 \int_{\R^{n_1}\times\R^{n_2}}\Psi_0\bigg(\frac{\mathbbm{1}_{(s,t]}(r_1)(\varphi(r_2)-\varphi_k(r_2))}{\epsilon}\bigg)\,d(r_1,r_2)<\epsilon
\]
so that the continuity follows. If $\operatorname{Leb}_{n_1}((s,t])>1$ then we write $\tilde{\epsilon}=\operatorname{Leb}_{n_1}((s,t])\epsilon$ and obtain
\begin{align*}
  \int_{\R^{n_1}\times\R^{n_2}}\Psi_0\bigg(\frac{\mathbbm{1}_{(s,t]}(r_1)(\varphi(r_2)-\varphi_k(r_2))}{\tilde{\epsilon}}\bigg)\,d(r_1,r_2)&\leq \int_{\R^{n_1}\times\R^{n_2}}\Psi_0\bigg(\frac{\mathbbm{1}_{(s,t]}(r_1)(\varphi(r_2)-\varphi_k(r_2))}{\epsilon}\bigg)\,d(r_1,r_2)\\
  &<\operatorname{Leb}_{n_1}((s,t])\epsilon=\tilde{\epsilon}
\end{align*}
for $k$ large enough. Again, the continuity follows.
\end{proof}

\begin{lemma}\label{lemma:ProductIndicator}
 Let $s,t\in\R$ with $s< t$. Let $\eta$ be a L\'evy white noise in $\R^n$ with L\'evy triplet $(\gamma,\sigma^2,\nu)$. Then the mapping
 \[
   \varphi\mapsto \langle\eta(\omega), \mathbbm{1}_{(s,t]}\otimes \varphi \rangle
 \]
is an element of $\mathscr{S}'(\R^{n-1})$ for almost all $\omega\in\Omega$.
\end{lemma}
\begin{proof}
 First, we note that $\mathbbm{1}_{(s,t]}\in B^{r}_{p,p}(\R)$ for all $p\in(1,\infty)$, $r\in(-\infty,\frac{1}{p})$. One way to see this is to use the equivalent norm
 \[
  \|\mathbbm{1}_{(s,t]}\|_{B^{r}_{p,p}(\R)}\simeq \| \mathbbm{1}_{(s,t]}\|_{L_p(\R)}+\left(\int_{|h|\leq 1}\int_{\R} |h|^{-1-rp} |\mathbbm{1}_{(s,t]}(x+h)-\mathbbm{1}_{(s,t]}(x)|^p\,dx\,dh\right)^{1/p}\quad(r\in(0,1)).
 \]
For this equivalence, we refer to \cite[Section 2.2.2 and Section 2.5.12]{Triebel_1983}. Since
\begin{align*}
    \int_{\R} |\mathbbm{1}_{(s,t]}(x+h)-\mathbbm{1}_{(s,t]}(x)|^p\,dx=\int_{\R} |\mathbbm{1}_{(s,t]}(x+h)-\mathbbm{1}_{(s,t]}(x)|\,dx=2|h|,
\end{align*}
for $|h|\leq t-s$, we only have to check for which parameters we have
\[
 \int_{|h|\leq1} |h|^{-rp}\,dh<\infty.
\]
This is the case if and only if $r<\frac{1}{p}$. If we now take $p=1+\epsilon$ for $\epsilon>0$ small, then we have that
\begin{align*}
  &B^{r}_{p,p}(\R)\otimes_{\alpha_p} S^{(r,\ldots,r)}_{p}H_p(\R^{n-1}) \stackrel{\eqref{Eq:Bessel_potential_embedding2}}{\hookrightarrow} H^r_p(\R) \otimes_{\alpha_p} S^{(r,\ldots,r)}_{p}H_p(\R^{n-1})\\
  &\qquad\stackrel{\eqref{Eq:Bessel_potential_tensor}}{\cong }S^{(r,\ldots,r)}_pH(\R^{n})\stackrel{\eqref{Eq:Bessel_potential_embedding1}}{\hookrightarrow} H^r_p(\R^{n})\stackrel{\eqref{Eq:Bessel_potential_embedding3}}{\hookrightarrow} B^{r-\epsilon}_{p,p}(\R^n),
\end{align*}
so that the mapping
\begin{align}\label{Eq:Schwartz_Indicator_Tensor_1}
 \mathscr{S}(\R^{n-1})\to B^{r-\epsilon}_{1+\epsilon,1+\epsilon}(\R^n),\,\varphi\mapsto (\mathbbm{1}_{(s,t]}\otimes \varphi)\langle\,\cdot\,\rangle^{N/p}
\end{align}
is continuous for arbitrary $N>0$, where $\langle\xi\rangle^{N/p}$ is taken for $\xi\in\R^n$. We take the detour via the Bessel potential scale since the embedding \eqref{Eq:Bessel_potential_embedding1} seems to be not available in the literature for Besov spaces. For the Bessel potential scale this holds as $S^{0,\ldots,0}_pH(\R^n)=L_p(\R^n)$ so that Fourier multiplier techniques are directly available, see the references given for \eqref{Eq:Bessel_potential_embedding1}.\\
Combining \eqref{Eq:Schwartz_Indicator_Tensor_1} with Proposition \ref{Prop:Weighted_Besov_Isomorphism} shows that also
\begin{align}\label{Eq:Schwartz_Indicator_Tensor}
 \mathscr{S}(\R^n)\to B^{r-\epsilon}_{1+\epsilon,1+\epsilon}(\R^n,\langle\cdot\rangle^{N}),\,\varphi\to \mathbbm{1}_{(s,t]}\otimes \varphi
\end{align}
is continuous. 
Now we combine \eqref{Eq:Schwartz_Indicator_Tensor} with Theorem \ref{Thm:White_Noise_Regularity}. Since the Blumenthal-Getoor indices are not larger than $2$, it holds that
\[
 \mathbb{P}(\eta\in B^{-3/4}_{1+\epsilon,1+\epsilon}(\R^n,\langle\cdot\rangle^{\rho}))=1
\]
for $\epsilon\in(0,1)$ and $\rho$ negative enough. Therefore, taking $r-\epsilon>3/4$ and $N$ large enough in \eqref{Eq:Schwartz_Indicator_Tensor} shows that
\[
 \lim_{n\to\infty} \langle\eta(\omega),\mathbbm{1}_{(s,t]}\otimes\varphi_n\rangle =0
\]
for almost all $\omega\in\Omega$ and all $(\varphi_{n})_{n\in\N}\subset\mathscr{S}(\R^n)$ such that $\lim_{n\to\infty}\varphi_n=0$. Hence, $\eta(\omega)$ is indeed a tempered distribution.
\end{proof}

\begin{proposition}\label{Prop:WhiteNoiseIndicatorTensor}
  Let $n=n_1+n_2$ with $n_1,n_2\in\N$ and $s,t\in\R^{n_1}$, $s\leq t$. Let further $\eta$ be a L\'evy white noise on $\R^n$ with L\'evy triplet $(\gamma,\sigma^2,\nu)$ on the complete probability space $(\Omega,\mathcal{F},\mathbb{P})$. Then the mapping
  \[
   \eta_{(s,t]}\colon(\Omega,\mathcal{F},\mathbb{P})\to(\mathscr{S}'(\R^{n_2}),\mathcal{B}_c(\mathscr{S}'(\R^{n_2}))),\,\omega\mapsto [\varphi\mapsto\langle\eta(\omega),\mathbbm{1}_{(s,t]}\otimes\varphi\rangle]
  \]
is a modification of a L\'evy white noise with L\'evy triplet $\operatorname{Leb}_{n_1}((s,t])(\gamma,\sigma^2,\nu)$. 
\end{proposition}
\begin{proof}
 It suffices to show the assertion for $n_1=1$. The general assertion then follows by iteration. So let $n_1=1$. Then Lemma \ref{lemma:ProductIndicator} shows that, $\eta_{(s,t]}(\omega)$ is indeed a tempered distribution almost surely. Thus, after changing it on a set of measure $0$, we can assume that it is a $\mathscr{S}'(\R^{n-1})$-valued mapping. For the measurability, it suffices to show that the preimages of cylindrical sets of the form
 \[
  C:=\{u\in\mathscr{S}'(\R^{n-1}):(\langle u,\varphi_1\rangle,\ldots, \langle u,\varphi_N\rangle)\in B\}
 \]
for some $N\in\N$ and some open set $B\subset\R^N$ under $\eta_{(s,t]}$ are elements of $\mathcal{F}$. So let $C$ be such a set. By Proposition \ref{Prop:IntegrableFunctions} and Lemma \ref{lemma:ProductIndicator_old}, we can take sequences $(\psi_{j,k})_{k\in\N}\subset\mathscr{D}(\R^n)$, $j=1,\ldots,N$ such that $\psi_{j,k}\to\mathbbm{1}_{(s,t]}\otimes\varphi_j$ in $L_0(\eta)$ as $k\to\infty$. Hence, we have that
\[
 \langle\eta,\psi_{j,k}\rangle\to\langle\eta,\mathbbm{1}_{(s,t]}\otimes\varphi_j\rangle\quad\text{in probability as }k\to\infty.
\]
By taking subsequences, we may without loss of generality assume that the convergence is also almost surely. Let $\tilde{K}\in\mathcal{F}$ be the set on which there is no pointwise convergence and 
\[
 K:=\tilde{K}\cap \eta_{(s,t]}^{-1}(C).
\]
Since the probability space is complete, it follows that $K\in\mathcal{F}$. Now we define
\begin{align*}
 B_l&:=\{x\in B\colon \operatorname{dist}(x,\R^n\setminus B)>\tfrac{1}{l}\},\\
 C_l&:=\{u\in\mathscr{S}'(\R^{n-1}):(\langle u,\varphi_1\rangle,\ldots, \langle u,\varphi_N\rangle)\in B_l\},\\
 \tilde{C}_l&:=\{u\in\mathscr{S}'(\R^{n-1}):(\langle u,\varphi_1\rangle,\ldots, \langle u,\varphi_N\rangle)\in \overline{B_l}\}
\end{align*}
for $l\in\N$. Note that we have 
\begin{align}\label{eq:NonConvergingPointsInclusion}
 K\subset \eta_{(s,t]}^{-1}(C)=\bigcup_{l\in\N}\eta_{(s,t]}^{-1}(C_l)=\bigcup_{l\in\N}\eta_{(s,t]}^{-1}(\tilde{C}_l).
\end{align}
Let further
\[
 A_{k,l}:=(\langle\eta,\psi_{1,k}\rangle,\ldots,\langle\eta,\psi_{N,k}\rangle)^{-1}(B_l)\setminus \tilde{K} \subset\Omega.
\]
Note that $A_{k,l}\in\mathcal{F}$ for all $k,l\in\N$ since $\eta$ is a a generalized random process. By construction, we have that $\liminf_{k\to\infty} A_{k,l}\in\mathcal{F}$ and that it consists of all $\omega\in\Omega$ such that
\[
 (\langle\eta(\omega),\psi_{1,k}\rangle,\ldots,\langle\eta(\omega),\psi_{N,k}\rangle)\to  (\langle\eta(\omega),\mathbbm{1}_{(s,t]}\otimes\varphi_1\rangle,\ldots,\langle\eta(\omega),\mathbbm{1}_{(s,t]}\otimes\varphi_N\rangle)\quad\text{as }k\to\infty
\]
and such that $ (\langle\eta(\omega,\psi_{1,k}\rangle,\ldots,\langle\eta(\omega),\psi_{N,k}\rangle)\in B_l$ for $k\in\N$ large enough. In particular, for $\omega\in\liminf_{k\to\infty} A_{k,l}$ it holds that $(\langle\eta(\omega),\mathbbm{1}_{(s,t]}\otimes\varphi_1\rangle,\ldots,\langle\eta(\omega),\mathbbm{1}_{(s,t]}\otimes\varphi_N\rangle)\in \overline{B}_l$ and thus
\[
 \liminf_{k\to\infty} A_{k,l}\subset (\langle \eta_{(s,t]},\varphi_1\rangle,\ldots,\langle \eta_{(s,t]},\varphi_N\rangle)^{-1}(\overline{B}_l)=\eta_{(s,t]}^{-1}(\tilde{C}_l).
\]
Together with \eqref{eq:NonConvergingPointsInclusion} this yields
\begin{align}\label{eq:LiminfIncluded}
 \bigcup_{l\in\N}\liminf_{k\to\infty} A_{k,l}\cup K\subset\eta_{(s,t]}^{-1}(C).
\end{align}
For the converse inclusion let $\omega\in \eta_{(s,t]}^{-1}(C)$ so that
\[
(\langle\eta(\omega),\mathbbm{1}_{(s,t]}\otimes\varphi_1\rangle,\ldots,\langle\eta(\omega),\mathbbm{1}_{(s,t]}\otimes\varphi_N\rangle)\in B.
\]
Since $B$ is open, there is an $l\in\N$ such that
\[
(\langle\eta(\omega),\mathbbm{1}_{(s,t]}\otimes\varphi_1\rangle,\ldots,\langle\eta(\omega),\mathbbm{1}_{(s,t]}\otimes\varphi_N\rangle)\in B_l.
\]
If $\omega\notin\tilde{K}$, then
\[
 (\langle\eta(\omega),\psi_{1,k}\rangle,\ldots,\langle\eta(\omega),\psi_{N,k}\rangle)\to  (\langle\eta(\omega),\mathbbm{1}_{(s,t]}\otimes\varphi_1\rangle,\ldots,\langle\eta(\omega),\mathbbm{1}_{(s,t]}\otimes\varphi_N\rangle)\quad\text{as }k\to\infty
\]
and since $B_l$ is open, it holds that 
\[
 (\langle\eta(\omega),\psi_{1,k}\rangle,\ldots,\langle\eta(\omega),\psi_{N,k}\rangle)\in B_l
\]
for $k$ large enough. Hence, it follows that $\omega\in A_{k,l}$ for $k$ large enough and therefore $\omega\in\liminf_{k\to\infty} A_{k,l}$. If in turn $\omega\in\tilde{K}$, then $\omega\in\tilde{K}\cap  \eta_{(s,t]}^{-1}(C)=K$. Hence, it follows that
\begin{align}\label{eq:IncludedInLiminf}
  \eta_{(s,t]}^{-1}(C)\subset \bigcup_{l\in\N}\liminf_{k\to\infty} A_{k,l}\cup K
\end{align}
Together with \eqref{eq:LiminfIncluded} it now follows that
\[
  \eta_{(s,t]}^{-1}(C)= \bigcup_{l\in\N}\liminf_{k\to\infty} A_{k,l}\cup K\in\mathcal{F}
\]
so that $\eta_{(s,t]}$ is indeed a generalized random process.
Finally, we show that $\eta_{(s,t]}$ is even a L\'evy white noise with L\'evy triplet $\operatorname{Leb}_{n_1}((s,t])(\gamma,\sigma^2,\nu)$ by simply computing its characteristic functional:  Let $\Psi$ be the L\'evy exponent of $\eta$. Then we obtain
 \begin{align*}
  \widehat{\mathbb{P}}_{\eta(\mathbbm{1}_{(s,t]}\otimes (\,\cdot\,))}(\varphi)&=\E[ e^{i\eta(\mathbbm{1}_{(s,t]}\otimes \varphi)}]=\exp\bigg(\int_{\R^n}\Psi([\mathbbm{1}_{(s,t]}\otimes \varphi](r))\,dr\bigg)\\
  &=\exp\bigg(\int_{\R^n} i\gamma [\mathbbm{1}_{(s,t]}\otimes \varphi](r)-\frac{\sigma^2[\mathbbm{1}_{(s,t]}^2\otimes \varphi^2](r)}{2}\\
  &\;\;+\int_{\R\setminus\{0\}} e^{ix[\mathbbm{1}_{(s,t]}\otimes \varphi](r)}-1+ix[\mathbbm{1}_{(s,t]}\otimes \varphi](r) \mathbbm{1}_{|x|\leq1}\,d\nu(x)\,dr\bigg)\\
    &=\exp\bigg(\op{Leb}_{1}((s,t])\int_{\R^{n-1}}i\gamma \varphi(r)-\frac{\sigma^2\varphi^2(r)}{2}\\
  &\;\;+\int_{\R\setminus\{0\}} e^{ix \varphi(r)}-1+ix\varphi(r) \mathbbm{1}_{|x|\leq1}\,d\nu(x)\,dr\bigg)\\
  &=\exp\bigg(\op{Leb}_{1}((s,t])\int_{\R^{n-1}}\Psi(\varphi(r))\,dr\bigg).
 \end{align*}
 Altogether, we obtain the assertion.
\end{proof}

\begin{lemma}\label{Lemma:BesovBorelInCylindrical}
 Let $T>0$, $\bar{s}=(s_1,\ldots,s_n)\in\R^n$, $s,\rho\in\R$, $p,q\in(1,\infty)$ and let $p',q'\in(1,\infty)$ be the conjugated H\"older indices.
 \begin{enumerate}[(a)]
  \item \label{Lemma:BesovBorelInCylindrical_Norming}There is a sequence $(\psi_k)_{k\in\N}\subset\mathscr{S}(\R^n)$ with $\|\psi\|_{B^{-s}_{p',q'}(\R^n,\langle\,\cdot\,\rangle^{\rho(1-p')})}=1$ such that
  \[
   \|u\|_{B^{s}_{p,q}(\R^n,\langle\,\cdot\,\rangle^{\rho})}=\sup_{k\in\N}|\langle u,\psi_k\rangle|
  \]
    for all $u\in B^{s}_{p,q}(\R^n,\langle\,\cdot\,\rangle^{\rho})$.
 \item There is a sequence $(\psi_k)_{k\in\N}\subset\mathscr{S}_0([0,T]^n)$ with $\|\psi\|_{S^{-\bar{s}}_{p',p'}B([0,T]^n)}=1$ such that
  \[
   \|u\|_{S^{\bar{s}}_{p,p}B([0,T]^n)}=\sup_{k\in\N}|\langle u,\psi_k\rangle|
  \]
    for all $u\in S^{\bar{s}}_{p,p}B([0,T]^n)$.
 \item \label{Lemma:BesovBorelInCylindrical_sigma_field} The Borel $\sigma$-field of $B^{s}_{p,q}(\R^n,\langle\,\cdot\,\rangle^{\rho})$ is contained in $\mathcal{B}_c(\mathscr{S}'(\R^n))$.
 \item The Borel $\sigma$-field of $S^s_{p,p}B([0,T]^n)$ is contained in $\mathcal{B}_c(\mathscr{S}'([0,T]^n))$.
 \end{enumerate}
\end{lemma}
\begin{proof}
 \begin{enumerate}[(a)]
  \item Since $B^{s}_{p,q}(\R^n)$ is separable (see for example \cite[Section 2.5.5, Remark 1]{Triebel_1983}) and since the spaces $B^{s}_{p,q}(\R^n,\langle\,\cdot\,\rangle^{\rho})$ and $B^{s}_{p,q}(\R^n)$ are isomorphic by Proposition \ref{Prop:Weighted_Besov_Isomorphism}, it follows that also $B^{s}_{p,q}(\R^n,\langle\,\cdot\,\rangle^{\rho})$ is separable. Hence, it follows from \cite[Proposition B.1.10]{HvNVW16} together with Proposition \ref{Prop:DualBesov} that there is a sequence $(\varphi_k)_{k\in\N}\subset B^{-s}_{p',q'}(\R^n,\langle\,\cdot\,\rangle^{\rho(1-p')})$ with $\|\varphi_k\|_{B^{-s}_{p',q'}(\R^n,\langle\,\cdot\,\rangle^{\rho(1-p')})}=1$ such that
  \[
   \| u \|_{B^{s}_{p,q}(\R^n,\langle\,\cdot\,\rangle^{\rho})}=\sup_{k\in\N}|\langle u,\varphi_k\rangle|.
  \]
Since $f\mapsto \langle\,\cdot\,\rangle^{-\rho/p}f$ leaves $\mathscr{S}(\R^n)$ invariant and since it is an isomorphism between $B^{s'}_{p',q'}(\R^n,\langle\,\cdot\,\rangle^{\rho(1-p')})$ and $B^{-s}_{p',q'}(\R^n)$, it follows from the density of $\mathscr{S}(\R^n)$ in $B^{-s}_{p',q'}(\R^n)$ (see for example \cite[Section 2.3.3]{Triebel_1983}) that we have the dense embedding
\[
 \mathscr{S}(\R^n)\stackrel{d}{\hookrightarrow}B^{s'}_{p',q'}(\R^n,\langle\,\cdot\,\rangle^{\rho(1-p')}).
\]
Therefore, there are sequences $(\varphi_{k,l})_{l\in\N}\subset \mathscr{S}(\R^n)$ with $\|\varphi_{k,l}\|_{B^{-s}_{p',q'}(\R^n,\langle\,\cdot\,\rangle^{\rho(1-p')})}=1$ such that $\varphi_{k,l}\to\varphi_k$ as $l\to\infty$. Thus, we obtain
  \[
   \| u \|_{B^{s}_{p,q}(\R^n,\langle\,\cdot\,\rangle^{\rho})}=\sup_{k,l\in\N}|\langle u,\varphi_{k,l}\rangle|.
  \]
Since $\N^2$ is countable, we can rename the functions and obtain the asserted sequence $(\psi_k)_{k\in\N}$.
\item The proof is almost the same as the one of Part (\ref{Lemma:BesovBorelInCylindrical_Norming}). One just has to use Proposition \ref{Prop:DualBesovDomain} instead of Proposition \ref{Prop:DualBesov}.
\item Since $B^{s}_{p,q}(\R^n,\langle\,\cdot\,\rangle^{\rho})$ is separable, its Borel $\sigma$-field is generated by the open balls. Hence, it suffices to show that for all $f\in B^{s}_{p,q}(\R^n,\langle\,\cdot\,\rangle^{\rho})$ and all $r>0$ we have $B(f,r)\in\mathcal{B}_c(\mathscr{S}'(\R^n))$. Now we use part (\ref{Lemma:BesovBorelInCylindrical_Norming}). Then we obtain that
\begin{align*}
   B(f,r)&=\bigcup_{m\in\N}\overline{B(f,r-\tfrac{1}{m})}
   = \bigcup_{m\in\N}\bigcap_{k\in\N}\{u\in B^{s}_{p,q}(\R^n,\langle\,\cdot\,\rangle^{\rho}): |\psi_{k}(f-u)|\leq r-\tfrac{1}{m} \}\\
   &= B^{s}_{p,q}(\R^n,\langle\,\cdot\,\rangle^{\rho})\cap \bigcup_{m\in\N}\bigcap_{k\in\N}\{u\in \mathscr{S}'(\R^n) : |\psi_{k}(f-u)|\leq r-\tfrac{1}{m} \}\\
   &=\bigcup_{m\in\N}\bigcap_{k\in\N}\{u\in \mathscr{S}'(\R^n) : |\psi_{k}(f-u)|\leq r-\tfrac{1}{m} \} \in \mc{B}_c(\ms{S}'(\R^n))
  \end{align*}
    In the last step we used that a tempered distribution $u_0$ is an element of $B^{s}_{p,q}(\R^n,\langle\,\cdot\,\rangle^{\rho})$ if $\|u_0\|_{B^{s}_{p,q}(\R^n,\langle\,\cdot\,\rangle^{\rho})}<\infty$. By part (\ref{Lemma:BesovBorelInCylindrical_Norming}) this is satisfied if 
    \[
     u_0\in \bigcap_{k\in\N}\{u\in \mathscr{S}'(\R^n) : |\psi_{k}(f-u)|\leq r-\tfrac{1}{m} \} .
    \]
This yields the assertion.
\item The proof is almost the same as the one of part (\ref{Lemma:BesovBorelInCylindrical_sigma_field}).
 \end{enumerate}
\end{proof}

\begin{lemma}\label{Prelim:Lemma:BesovValuedProcess}
 Let $\eta$ be a L\'evy white noise. For $t_2\geq t_1\geq 0$ let again
 \[
  \eta_{(t_1,t_2]}\colon\Omega\to\mathscr{S}'(\R^{n-1}), \omega\mapsto \langle \eta(\omega),\mathbbm{1}_{(t_1,t_2]}\otimes(\,\cdot\,)\rangle.
 \]
Suppose that for all $t\geq0$ the mapping $\eta_{(0,t]}$ takes values in the Besov space $B^{s}_{p,q}(\R^{n-1},\langle\,\cdot\,\rangle^{\rho})$ for fixed parameters $s,\rho\in\R$ and $1<p,q<\infty$. Let $B^{s}_{p,q}(\R^{n-1},\langle\,\cdot\,\rangle^{\rho})$ be endowed with its Borel $\sigma$-field. Then $(\eta_{(0,t]})_{t\geq0}$ is a $B^{s}_{p,q}(\R^{n-1},\langle\,\cdot\,\rangle^{\rho})$-valued stochastic process with stationary and independent increments. The same assertion holds if $\eta$ is restricted to $[0,T]^n$ and if $B^{s}_{p,q}(\R^{n-1},\langle\,\cdot\,\rangle^{\rho})$ is replaced by $S^{\bar{s}'}_{p,p}B([0,T]^{n-1})$ for some $\bar{s}'=(s_2,\ldots,s_n)\in\R^{n-1}$.
\end{lemma}
\begin{proof}
 We only show the assertion for $B^{s}_{p,q}(\R^{n-1},\langle\,\cdot\,\rangle^{\rho})$. The proof for $S^{\bar{s}}_{p,p}B([0,T]^{n-1})$ can be carried out the same way.\\
 It follows from Proposition \ref{Prop:WhiteNoiseIndicatorTensor} that the mappings
 \[
   \eta_{(0,t]}\colon(\Omega,\mathcal{F})\to(\mathscr{S}'(\R^{n-1}),\mathcal{B}_c(\mathscr{S}'(\R^{n-1})))
 \]
are measurable. Thus, Lemma \ref{Lemma:BesovBorelInCylindrical} shows that the mappings $\eta_{(0,t]}$ ($t\geq0$) are $B^{s}_{p,q}(\R^{n-1},\langle\,\cdot\,\rangle^{\rho})$-valued random variables, where $B^{s}_{p,q}(\R^{n-1},\langle\,\cdot\,\rangle^{\rho})$-valued is endowed with the Borel $\sigma$-field. Proposition \ref{Prop:WhiteNoiseIndicatorTensor} also shows that the increments are stationary. Hence, it only remains to prove that the increments are independent.\\
Since the sets $\prod_{j=1}^N(-\infty,\alpha_j]$ with $\alpha_j\in\R$ generate the Borel $\sigma$-field in $\R^N$, we also have that sets of the form
\[
 \{u\in\mathscr{S}'(\R^{n-1})\vert \forall j\in\{1,\ldots,N\}:\langle u,\varphi_j\rangle\leq \alpha_j\}
\]
for some $N\in\N$, $\varphi_1,\ldots,\varphi_N\in\mathscr{S}(\R^{n-1})$ and $\alpha_1,\ldots,\alpha_N\in\R$ generate $\mathcal{B}_c(\mathscr{S}(\R^{n-1}))$. But together with Lemma \ref{Lemma:BesovBorelInCylindrical} this implies that the Borel $\sigma$-field of $B^{s}_{p,q}(\R^{n-1},\langle\,\cdot\,\rangle^{\rho})$ is generated by sets of the form 
\[
 \{u\in B^{s}_{p,q}(\R^{n-1},\langle\,\cdot\,\rangle^{\rho})\vert \forall j\in\{1,\ldots,N\}:\langle u,\varphi_j\rangle\leq \alpha_j\}.
\]
This collection of sets is stable under finite intersections and thus, it suffices to verify the independece for preimages of such sets. As in Remark \ref{Rem:WhiteNoiseModel} one can use the characteristic function from Proposition \ref{Prop:IntegrableFunctions} to show that the random variables $\langle \eta_{(t_1,t_0]},\varphi_1\rangle,\ldots,\langle \eta_{(t_N,t_{N-1}]},\varphi_N\rangle$ are independent for all choices of $N\in\N$, $0\leq t_0<t_1\ldots<t_N$, $\varphi_1,\ldots,\varphi_N\in\mathscr{S}(\R^{n-1})$. Thus, for all choices of $M, N_1,\ldots, N_M\in\N$, $0\leq t_0<t_1\ldots<t_M$, $\alpha_{j,k}\in\R$ and  $\varphi_{j,k}\in\mathscr{S}(\R^{n-1})$ $1\leq j\leq M$, $1\leq k\leq N_j$ we have that
{\allowdisplaybreaks
 \begin{align*}
  &\mathbb{P}\bigg(\bigcap_{j=1}^M \big\{\eta_{(t_{j-1},t_j]}\in\{u\in B^{s}_{p,q}(\R^{n-1},\langle\,\cdot\,\rangle^{\rho})): \forall k\in\{1,\ldots,N_j\}:\langle u,\varphi_{j,k}\rangle\leq \alpha_{j,k}\}\big\}\bigg)\\
  =&\mathbb{P}\bigg(\bigcap_{j=1}^M\bigcap_{k=1}^{N_j} \big\{\eta_{(t_{j-1},t_j]}\in\{u\in B^{s}_{p,q}(\R^{n-1},\langle\,\cdot\,\rangle^{\rho})): \langle u,\varphi_{j,k}\rangle\leq \alpha_{j,k}\}\big\}\bigg)\\
  =&\mathbb{P}\bigg(\bigcap_{j=1}^M\bigcap_{k=1}^{N_j} \big\{\langle\eta_{(t_{j-1},t_j]},\varphi_{j,k}\rangle\leq \alpha_{j,k}\big\}\bigg)\\
  =&\prod_{j=1}^M\mathbb{P}\bigg(\bigcap_{k=1}^{N_j} \big\{\langle\eta_{(t_{j-1},t_j]},\varphi_{j,k}\rangle\leq \alpha_{j,k}\big\}\bigg)\\
  =&\prod_{j=1}^M\mathbb{P}\bigg(\big\{\eta_{(t_{j-1},t_j]}\in\{u\in B^{s}_{p,q}(\R^{n-1},\langle\,\cdot\,\rangle^{\rho})): \forall k\in\{1,\ldots,N_j\}:\langle u,\varphi_{j,k}\rangle\leq \alpha_{j,k}\}\big\}\bigg).
 \end{align*}}
 This shows that $(\eta_{(0,t]})_{t\geq0}$ has independent increments.
\end{proof}

\begin{theorem}\label{Prelim:Lemma:ContinuityProbability}
 Consider the situation of Lemma \ref{Prelim:Lemma:BesovValuedProcess} with $s<0$. Let $\alpha,r\in\R$ and $1\leq p_1\leq p'\leq p_2<\infty$ where $p'$ denotes the conjugated H\"older index of $p\in(1,\infty)$.
 \begin{enumerate}[(a)]
  \item Suppose that we have the embedding $L_{p_1}([0,T]\times\R^{n-1})\hookrightarrow L_0(\eta)$ for all $T>0$ and that $\alpha<\min\{\rho,p(n-1)\big(\tfrac{1}{p'}-\tfrac{1}{p_1}\big)\}$. Then $(\eta_{(0,t]})_{t\geq0}$ is a $B^{s}_{p,q}(\R^{n-1},\langle\,\cdot\,\rangle^{\alpha})$-valued L\'evy process.
  \item Suppose that we have the embedding $L_{p_2}([0,T]\times\R^{n-1})\hookrightarrow L_0(\eta)$ for all $T>0$ and $r<\min\{s, (n-1)\big(\tfrac{1}{p_2}-\tfrac{1}{p'}\big)\}$. Then $(\eta_{(0,t]})_{t\geq0}$ is a $B^{r}_{p,q}(\R^{n-1},\langle\,\cdot\,\rangle^{\rho})$-valued L\'evy process.
  \item Suppose that we have the embedding $L_{p_1}([0,T]\times\R^{n-1})\cap L_{p_2}([0,T]\times\R^{n-1})\hookrightarrow L_0(\eta)$ for all $T>0$ as well as the estimates $\alpha<\min\{\rho,p(n-1)\big(\tfrac{1}{p'}-\tfrac{1}{p_1}\big)\}$ and $r<\min\{s, (n-1)\big(\tfrac{1}{p_2}-\tfrac{1}{p'}\big)\}$. Then $(\eta_{(0,t]})_{t\geq0}$ is a $B^{r}_{p,q}(\R^{n-1},\langle\,\cdot\,\rangle^{\alpha})$-valued L\'evy process.
 \end{enumerate}
\end{theorem}
\begin{proof}
 \begin{enumerate}[(a)]
  \item Let $\alpha$ be chosen as in the assertion. Then we have the embedding
  \begin{align*}
   B^{s}_{p,q}(\R^{n-1},\langle\,\cdot\,\rangle^{\rho})\stackrel{d}{\hookrightarrow} B^{s}_{p,q}(\R^{n-1},\langle\,\cdot\,\rangle^{\alpha})
  \end{align*}
so that $\eta_{(0,t]}$ takes values in $B^{s}_{p,q}(\R^{n-1},\langle\,\cdot\,\rangle^{\alpha})$ for $t\geq0$. By Proposition \ref{Prop:DualBesov}, the dual space of $B^{s}_{p,q}(\R^{n-1},\langle\,\cdot\,\rangle^{\alpha})$ is given by $B^{-s}_{p',q'}(\R^{n-1},\langle\,\cdot\,\rangle^{\alpha(1-p')})$. It follows from Lemma \ref{Lemma:BesovBorelInCylindrical} that there is a sequence $(\psi_k)_{k\in\N}\subset \mathscr{S}(\R^{n-1})$ with $\|u\|_{B^{s}_{p,q}(\R^{n-1},\langle\,\cdot\,\rangle^{\alpha})}=\sup_{k\in\N}|\psi_k(u)|$ and $\|\psi_k\|_{B^{-s}_{p',q'}(\R^{n-1},\langle\,\cdot\,\rangle^{\alpha(1-p')})}=1$.
Using \cite[Proposition 3]{Fageot_Fallah_Unser_2017} and the elementary embedding $B^{\tilde{s}}_{p_1,q}(\R^{n-1})\hookrightarrow L_{p_1}(\R^{n-1})$ for $\tilde{s}>0$, we also obtain that
\begin{align*}
  B^{-s}_{p',q'}(\R^{n-1},\langle\,\cdot\,\rangle^{\alpha(1-p')})&\hookrightarrow L_{p_1}(\R^{n-1})\quad \text{if } \alpha<p(n-1)\big(\tfrac{1}{p'}-\tfrac{1}{p_1}\big).
\end{align*}
In this case, we have that $(\psi_k)_{k\in\N}$ is bounded in $L_{p_1}(\R^{n-1})$ with norms not larger than $1$.  Therefore, if $t,t_0\in[0,T]$ then $(\mathbbm{1}_{(0,t_0]}-\mathbbm{1}_{(0,t]})\otimes \psi_k$ goes uniformly in $k\in\N$ to $0$ in $L_{p_1}([0,T]\times\R^{n-1})$ as $t\to t_0$. But since we have the continuous embeddings
  \[
   L_{p_1}([0,T]\times\R^{n-1})\stackrel{\operatorname{id}}{\hookrightarrow} L(\eta) \stackrel{\eta}{\hookrightarrow} L_0(\Omega,\mathcal{F},\mathbb{P}),
  \]
  it follows that  $\eta((\mathbbm{1}_{(0,t_0]}-\mathbbm{1}_{(0,t]})\otimes \psi_k)$ goes uniformly in $k\in\N$ to $0$ in probability as $t\to t_0$. Now Lemma \ref{Lemma:BesovBorelInCylindrical} shows that $\eta((\mathbbm{1}_{(0,t_0]}-\mathbbm{1}_{(0,t]})\otimes \,\cdot\,)$ goes to $0$ in probability with respect to the space $B^{s}_{p,q}(\R^{n-1},\langle\,\cdot\,\rangle^{\alpha})$ as $t\to t_0$. Together with Lemma \ref{Prelim:Lemma:BesovValuedProcess} proves the assertion.  
\item This can be shown with the same proof as part (a). One just has to replace $B^{s}_{p,q}(\R^{n-1},\langle\,\cdot\,\rangle^{\alpha})$ by $B^{r}_{p,q}(\R^{n-1},\langle\,\cdot\,\rangle^{\rho})$ and  $L_{p_1}([0,T]\times\R^{n-1})$ by $ L_{p_2}([0,T]\times\R^{n-1})$. In this case, we have
\begin{align*}
  B^{-r}_{p',q'}(\R^{n-1},\langle\,\cdot\,\rangle^{\rho(1-p')})&\hookrightarrow L_{p_2}(\R^{n-1})\quad \text{if } r< (n-1)\big(\tfrac{1}{p_2}-\tfrac{1}{p'}\big).
\end{align*}
  Except for these changes, the proof can be carried out in the same way.
\item Also this case can be carried out as part (a). One just has to replace $B^{s}_{p,q}(\R^{n-1},\langle\,\cdot\,\rangle^{\alpha})$ by $B^{r}_{p,q}(\R^{n-1},\langle\,\cdot\,\rangle^{\alpha})$ and  $L_{p_1}([0,T]\times\R^{n-1})$ by $ L_{p_1}([0,T]\times\R^{n-1})\cap L_{p_2}([0,T]\times\R^{n-1})$. Of course, in this case both estimates on $r$ and $\alpha$ have to be satisfied.
 \end{enumerate}
\end{proof}

\begin{theorem}\label{Prelim:Lemma:ContinuityProbability_Mixed}
 Consider the situation of Lemma \ref{Prelim:Lemma:BesovValuedProcess} with $\bar{s}'=(s_2,\ldots,s_{n})\in(-\infty,0)^{n-1}$ and $1<p<\infty$. Let $p'$ be the conjugated H\"older index and $1\leq p_1<\infty$ such that $L_{p_1}([0,T]^{n})\hookrightarrow L_0(\eta)$. If $\max\{s_2,\ldots,s_n\}<\frac{1}{p_1}-\frac{1}{p'}$, then the restriction of $(\eta_{(0,t]})_{t\geq0}$ to $[0,T]^{n-1}$ is a $S^{\bar{s}'}_{p,p}B([0,T]^{n-1})$-valued L\'evy process.
\end{theorem}
\begin{proof}
 The proof is similar as the one of Theorem \ref{Prelim:Lemma:ContinuityProbability}. This time we use that
 \begin{align*}
  (S^{\bar{s}}_{p,p}B([0,T]^{n-1}))'&=S^{-\bar{s}}_{p',p',0}B([0,T]^{n-1})=B^{-s_2}_{p',p',0}([0,T])\otimes_{\alpha_p}\ldots \otimes_{\alpha_p} B^{-s_n}_{p',p',0}([0,T])\ldots)\\
  &\hookrightarrow L_{p_1}([0,T])\otimes_{\alpha_p}\ldots\otimes_{\alpha_p} L_{p_1}([0,T])\cong L_{p_1}([0,T]^{n-1}),
 \end{align*}
where $\otimes_{\alpha_p}$ denotes the tensor product with respect to the $p$-nuclear tensor norm and where we used that
\[
 B^{-\max\{s_2,\ldots,s_n\}}_{p',p',0}([0,T])\hookrightarrow B^{-\max\{s_2,\ldots,s_n\}}_{p',p'}([0,T]) \hookrightarrow B^{\epsilon}_{p_1,p_1}([0,T];E)\hookrightarrow L_{p_1}([0,T];E)
\]
if $-\max\{s_2,\ldots,s_n\}-\frac{1}{p'}>\epsilon-\frac{1}{p_1}$ and $\epsilon>0$. Here, the first embedding follows directly from the definitions. For the second embedding, we refer to \cite[Section 3.3.1]{Triebel_1983}. The last embedding can for example be found in \cite[Section 2.3.2, Remark 3]{Triebel_1992}. With the embedding
\[
 (S^{\bar{s}}_{p,p}B([0,T]^{n-1}))'\hookrightarrow L_{p_1}([0,T]^{n-1})
\]
at hand, the proof is analogous to the one of Theorem \ref{Prelim:Lemma:ContinuityProbability}.
\end{proof}

\begin{remark}
 Embeddings of the form 
 \[
  L_{p_1}(\R^n)\hookrightarrow L_0(\eta),\quad L_{p_2}(\R^n)\hookrightarrow L_0(\eta)\quad\text{or} \quad L_{p_1}(\R^n)\cap L_{p_2}(\R^n)\hookrightarrow L_0(\eta)
 \]
are satisfied for many different kinds of L\'evy white noise, see \cite[Table 1]{Fageot_Fallah_Unser_2017}. Accordingly, Theorem \ref{Prelim:Lemma:ContinuityProbability} and Theorem \ref{Prelim:Lemma:ContinuityProbability_Mixed} can be applied to them. As an example, we carry out the Gaussian case:
\end{remark}

\begin{corollary}\label{Prelim:Corollary:GaussianSpaceTimeRegularity}
 Consider the situation of Theorem \ref{Prelim:Lemma:ContinuityProbability} and suppose that the L\'evy triplet is given by $(0,1,0)$ so that we have the Gaussian case. Then the process $(\eta_{(0,t]})_{t\geq0}$ has a modification that is a Brownian motion with values in $B^{s}_{p,p}(\R^{n-1},\langle\,\cdot\,\rangle^{\rho})$ if 
 \[
  s<-\frac{n-1}{2},\quad \rho<-n+1.
 \]
\end{corollary}
\begin{proof}
 It follows from Proposition \ref{Prop:WhiteNoiseIndicatorTensor} and Theorem \ref{Thm:White_Noise_Regularity} that $\eta_{(0,t]}$ takes almost surely values in the weighted Besov space $B^{\widetilde{s}}_{p,p}(\R^{n-1},\langle\,\cdot\,\rangle^{\widetilde{\rho}})$ if
  \[
  \widetilde{s}<-\frac{n-1}{2},\quad \widetilde{\rho}<-n+1.
 \]
 By Remark \ref{Rem:Remark_Integrable_Functions} we know that $L_0(\eta)=L_2(\R^n)$. Hence, if $1<p\leq2$ we can consider case (a) in Lemma \ref{Prelim:Lemma:ContinuityProbability} with $p_1=2$. In this case, $\rho$ has to satisfy
 \[
  \rho<\min\{\widetilde{\rho},p(n-1)\big(\tfrac{1}{p'}-\tfrac{1}{p_1}\big)\}=\min\{\widetilde{\rho},(n-1)\big(\tfrac{p}{2}-1\big)\}<-n+1.
 \]
If in turn $2\leq p<\infty$, then we can use Theorem \ref{Prelim:Lemma:ContinuityProbability} (b) with $p_2=2$ so that we obtain the condition
 \[
  s<\min\{\widetilde{s}, (n-1)\big(\tfrac{1}{p_2}-\tfrac{1}{p'}\big)\}<-\frac{n-1}{2}.
 \]
 Altogether, we obtain the assertion.
\end{proof}

\begin{proposition}\label{Prelim:Prop:WhiteNoiseHDI}
 Let $n_1,n_2\in\N$ with $n_1+n_2=n$, $T>0$ and $\mathcal{O}\in\{[0,T]^{n_2},\R^{n_2}\}$. Suppose that there are $p,p_1,p_2\in[1,\infty)$ such that $L_{p_1}([0,T]^{n_1}\times \mathcal{O})\cap L_{p_2}([0,T]^{n_1}\times \mathcal{O})\hookrightarrow L_p(\eta)$ (for example $p=2$ in the symmetric case, see Remark \ref{Rem:Remark_Integrable_Functions}). Then the mapping
\[
 \partial_t \eta_{(0,t]}\colon \mathscr{S}_0([0,T]^{n_1})\otimes \mathscr{S}_0(\mathcal{O})\to L_p(\Omega,\mc{F},\mathbbm{P}),\;\psi\otimes \varphi\mapsto [\partial_1\ldots,\partial_{n_1} \langle \eta_{(0,t]}, \varphi\rangle ](\psi)
\]
extends again to the white noise $\eta$. Here, $[\partial_1\ldots,\partial_{n_1}  \langle \eta_{(0,t]}, \varphi\rangle ](\psi)$ means that we apply the distributional derivatives of the trajectories of $(\langle \eta_{(0,t]}, \varphi\rangle)_{t\geq0}$ to the test function $\psi$.
\end{proposition}
\begin{proof}
 By Proposition \ref{Prop:WhiteNoiseIndicatorTensor} it suffices to prove the result for $n_1=1$. Higher dimensions then follow by iteration. So let $\psi\in\mathscr{S}_0([0,T])$ and $\varphi\in\mathscr{S}_0(\mathcal{O})$. First, we define the function
 \[
  K\colon [0,\infty)\to L_{p_1}([0,T])\cap L_{p_2}([0,T]), t\mapsto [s\mapsto \mathbbm{1}_{[0,s)}(t) \psi'(t)].
 \]
This function is continuous and therefore Bochner integrable. Indeed, let $(t_k)_{k\in\N}\subset[0,\infty)$ such that $t_0=\lim_{k\to\infty} t_k$. Then for all $s\neq t_0$ we have that
\[
 \mathbbm{1}_{[0,s)}(t_k)\psi'(t_k)\to\mathbbm{1}_{[0,s)}(t_0)\psi'(t_0)\quad(k\to\infty)
\]
so that the continuity follows by dominated convergence. Moreover, we note that $f~\otimes~\varphi\mapsto \eta(f\otimes\varphi)$ defines a bounded linear operator from $L_{p_1}([0,T]\times \mathcal{O})\cap L_{p_2}([0,T]\times \mathcal{O})$ to $L_p(\Omega,\mc{F},\mathbbm{P})$ by Proposition \ref{Prop:IntegrableFunctions}. Using these two facts we may interchange the order of $\eta$ and the integration in the following computation:
{\allowdisplaybreaks
\begin{align*}
 [\partial_t\eta_{(0,t]}\varphi](\psi)&=[\partial_t\eta(\mathbbm{1}_{(0,t]}\otimes\varphi)](\psi)\\
 &=  (\psi(T)-\psi(0))\eta(\mathbbm{1}_{[0,T]}\otimes\varphi)-\int_0^T\eta(\mathbbm{1}_{(0,t]}\otimes\varphi)\psi'(t)\,dt\\
 &= \int_0^T\eta\big((\mathbbm{1}_{[0,T]}-\mathbbm{1}_{(0,t]})\otimes\varphi\big)\psi'(t)\,dt
 = \int_0^T\eta(\mathbbm{1}_{[0,\,\cdot\,)}(t)\otimes\varphi)\psi'(t)\,dt\\
 &= \int_0^T \eta (K(t)\otimes\varphi)\,dt=\eta\bigg(\int_0^T K(t)\,dt\otimes\varphi\bigg)\\
 &=\eta\bigg(\int_0^T \mathbbm{1}_{[0,\,\cdot\,)}(t)\psi'(t)\,dt\otimes\varphi\bigg)=\eta\bigg(\int_0^{\cdot} \psi'(t)\,dt\otimes\varphi\bigg)\\
 &=\eta(\psi\otimes \varphi).
\end{align*}}
As the tensor product $\mathscr{S}_0([0,T])\otimes \mathscr{S}_0(\mathcal{O})$ is sequentially dense in $\mathscr{S}_0([0,T]\times \mathcal{O})$ (see for example \cite[Theorem 1.8.1]{Amann_2003}), it follows from the continuity of $\eta\colon\ms{S}_0([0,T]\times \mathcal{O})\to L_p(\Omega,\ms{F},\mathbbm{P})$ that $\partial_t\eta_{(0,t]}$ extends to $\eta$.
\end{proof}

 \begin{theorem}
  \label{Prelim:Theorem:SpaceTimeWhiteNoiseRegularity}
  Let $0<T<\infty$ and let $\widetilde{\eta}$ be a L\'evy white noise restricted to $[0,T]\times\R^{n-1}$ with L\'evy triplet $(\gamma,\sigma^2,\nu)$, Blumenthal-Getoor indices $0\leq\underline{\beta}{}_{\infty}\leq\beta_{\infty}\leq 2$ and moment index $0<p_{max}\leq \infty$. Let further $p\in(1,\infty)$ and $\widetilde{p}\in(1,\infty)$ be fixed.
  \begin{enumerate}[(a)]
   \item The Gaussian case:\\
    Suppose that $\gamma=0$ and $\nu=0$. If $t\leq -\frac{1}{2}$, $s<-\frac{n-1}{2}$ and $\rho<-n+1$, then $\widetilde{\eta}$ has a modification $\eta$ such that $$\mathbb{P}\big(\eta\in B^{t}_{\widetilde{p},\infty}([0,T],B^{s}_{p,p}(\R^{n-1},\langle\,\cdot\,\rangle^{\rho}))\big)=1.$$ If $t> -\frac{1}{2}$, $s\geq-\frac{n-1}{2}$ or $\rho\geq-\frac{n-1}{p}$, then we have $$\mathbb{P}\big(\eta\notin B^{t}_{\widetilde{p},\infty}([0,T],B^{s}_{p,p}(\R^{n-1},\langle\,\cdot\,\rangle^{\rho}))\big)=1.$$ 
    \item The compound Poisson case:\\
    Let $p\in(1,\infty)$ and $1\leq p_1<p'<p_2<\infty$ such that $$L_{p_1}([0,T]\times \R^{n-1})\cap L_{p_2}([0,T]\times \R^{n-1})\hookrightarrow L_1(\eta).$$ Let further $t\leq -1$ and $t<-1$ if $p<2$, $s<(n-1)(\frac{1}{p}-1)$ and $\rho<-\frac{(n-1)p}{\min\{p,p_{max}\}}$. Then $\widetilde{\eta}$ has a modification $\eta$ such that $$\mathbb{P}\big(\eta\in B^{t}_{\widetilde{p},\widetilde{p}}([0,T],B^{s}_{p,p}(\R^{n-1},\langle\,\cdot\,\rangle^{\rho}))\big)=1.$$ 
    \item The general non-Gaussian case: \\
        Let $p\in(1,2]\cup2\N$ and $1\leq p_1<p'<p_2<\infty$ such that $$L_{p_1}([0,T]\times \R^{n-1})\cap L_{p_2}([0,T]\times \R^{n-1})\hookrightarrow L_1(\eta).$$ Let further $t\leq -1$ and $t<-1$ if $\tilde{p}< 2$, $s<(n-1)(\frac{1}{\max\{p,\beta_{\infty}\}}-1)$ and $\rho<-\frac{(n-1)p}{\min\{p,p_{max}\}}$. Then $\widetilde{\eta}$ has a modification $\eta$ such that $$\mathbb{P}\big(\eta\in B^{t}_{\widetilde{p},\widetilde{p}}([0,T],B^{s}_{p,p}(\R^{n-1},\langle\,\cdot\,\rangle^{\rho}))\big)=1.$$ 
  \end{enumerate}
 \end{theorem}
\begin{proof}
Proposition \ref{Prop:WhiteNoiseIndicatorTensor} yields that for all $t_0\in[0,T]$ we have that $\eta(\mathbbm{1}_{(0,t_0]}\otimes\,\cdot\,)$ is a white noise with the L\'evy triplet $(t_0\gamma,t_0\sigma^2,t_0\nu)$. Hence, for fixed $t_0\in[0,T]$ we can use Theorem \ref{Thm:White_Noise_Regularity} in order to obtain that $t_0\mapsto \eta(\mathbbm{1}_{(0,t_0]}\otimes\,\cdot\,)$ almost surely takes values in $B^{s}_{p,p}(\R^{n-1},\langle\,\cdot\,\rangle^{\rho}))$ with certain $s$ and $\rho$, depending on the respective case. Moreover,  it follows from Proposition \ref{Prelim:Prop:WhiteNoiseHDI} that we can write $\eta=\partial_t\eta(\mathbbm{1}_{(0,t]}\otimes\,\cdot\,)$. Hence, the Gaussian case follows from Corollary \ref{Prelim:Corollary:GaussianSpaceTimeRegularity} together with the regularity results on Brownian motions, Theorem \ref{Thm:Brownian_Paths}. The compound Poisson and the general non-Gaussian case follow from Theorem \ref{Prelim:Lemma:ContinuityProbability} together with Proposition \ref{Prop:Levy_Paths}. In order to see this, we note that
            \[
             (n-1)(\tfrac{1}{p}-1)\leq (n-1)(\tfrac{1}{p_2}-\tfrac{1}{p'}),\quad -\tfrac{(n-1)p}{\min\{p,p_{max}\}}\leq p(n-1)\big(\tfrac{1}{p'}-\tfrac{1}{p_1}\big).
            \]
            Hence, the estimates from Theorem \ref{Prelim:Lemma:ContinuityProbability} do not give additional restrictions.
\end{proof}

\begin{remark}
	 As in Remark \ref{Rem:Dropping_Some_Conditions} one can weaken the conditions on $p$ in the non-Gaussian case of Theorem \ref{Prelim:Theorem:SpaceTimeWhiteNoiseRegularity}. More precisely, the assertion of the non-Gaussian case of Theorem \ref{Prelim:Theorem:SpaceTimeWhiteNoiseRegularity} also holds if $p_{\max}\in2\N$ and $p\in(1,\infty)$ or if $p_{\max}\in(N,N+2)$ and $p\in (1,\infty)\setminus(N,N+2)$ for some $N\in2\N$.
\end{remark}

\begin{theorem}   \label{Prelim:Theorem:SpaceTimeWhiteNoiseRegularity_Mixed}Let $\epsilon,T>0$ and let $\tilde{\eta}$ be a L\'evy white noise restricted to $[0,T]^n$ and let $p\in(1,\infty)$. Let further $l=n$, i.e. the smoothness parameters of spaces with dominating mixed smoothness are elements of $\R^n$.
 \begin{enumerate}[(a)]
  \item The Gaussian case:\\
    There is a modification $\eta$ of $\tilde{\eta}$ such that for any 
        \[
         \mathbb{P}(\eta\in S^{(-\tfrac{1}{2}-\epsilon,\ldots,-\tfrac{1}{2}-\epsilon)}_{p,p}([0,T]^n))=1.
        \]
        Moreover, it holds that
        \[
         \mathbb{P}(\eta\in S^{(-\tfrac{1}{2},\ldots,-\tfrac{1}{2})}_{p,p}([0,T]^n))=0.
        \]
  \item The compound Poisson case: Let $p\in(1,\infty)$ and $1\leq p_1<p_2<\infty$ such that $$L_{p_1}([0,T]^n)\cap L_{p_2}([0,T]^n)\hookrightarrow L_1(\eta).$$ Let further $t\leq -1$ and $t<-1$ if $p<2$ and $s<\frac{1}{p}-1$. Then $\widetilde{\eta}$ has a modification $\eta$ such that $$\mathbb{P}\big(\eta\in S^{(t,\ldots,t,s)}_{p,p}([0,T]^n)\big)=1.$$ 
  \item The general non-Gaussian case:\\
  Let $p\in(1,\infty)$ and $1\leq p_1<p_2<\infty$ such that $$L_{p_1}([0,T]\times \R^{n-1})\cap L_{p_2}([0,T]\times \R^{n-1})\hookrightarrow L_1(\eta).$$
  Let further $t\leq -1$ and $t<-1$ if $p< 2$ and $s<\frac{1}{\max\{p,\beta_{\infty}\}}-1$.
  Then there is a modification $\eta$ of $\tilde{\eta}$ such that
        \[
         \mathbb{P}(\eta\in S^{(t,\ldots,t,s)}_{p,p}([0,T]^n))=1.
        \]
 \end{enumerate}
\end{theorem}
\begin{proof}
 First, we apply Proposition \ref{Prop:WhiteNoiseIndicatorTensor} with $n_1=n-1$. Thus, for fixed $t\in [0,T]^{n-1}$ we have that $\tilde{\eta}_{(0,t]}$ is a one-dimensional L\'evy white noise on $[0,T]$ which almost surely takes values in $B^{\tfrac{1}{2}-\epsilon}_{p,p}([0,T])$ in the Gaussian case or if $p<2$ and in $B^{\tfrac{1}{p}-1-\epsilon}_{p,p}([0,T])$ in the non-Gaussian case with $p>2$ by Theorem \ref{Thm:White_Noise_Regularity}. Now, it follows from Theorem \ref{Prelim:Lemma:ContinuityProbability_Mixed} that for fixed $(t_1,\ldots,t_{n-2})\in[0,T]^{n-2}$ the family $(\tilde{\eta}_{(0,(t_1,\ldots,t_{n-1})]})_{t_{n-1}\in[0,T]}$ is a $B^{\tfrac{1}{2}-\epsilon}_{p,p}([0,T])$-valued Brownian motion in the Gaussian case and a $B^{s}_{p,p}([0,T])$-valued L\'evy process in the other two cases. By Theorem \ref{Thm:Brownian_Paths} it has a modification which almost surely has paths in $B^{\frac{1}{2}}_{p,\infty}([0,T];B^{\tfrac{1}{2}-\epsilon}_{p,p}([0,T]))$ but not better in the Gausssian case and in $B^{t}_{p,p}([0,T];B^{s}_{p,p}([0,T]))$ in the non-Gaussian cases. Together with
 \[
  B^{\frac{1}{2}}_{p,\infty}([0,T];B^{\tfrac{1}{2}-\epsilon}_{p,p}([0,T]))\hookrightarrow B^{\frac{1}{2}-\epsilon}_{p,p}([0,T];B^{\tfrac{1}{2}-\epsilon}_{p,p}([0,T]))
 \]
and Proposition \ref{Cor:Isomorphies_Dominating_Mixed_Domain} we obtain the assertion for $n=2$. For general $n\in\N$ we iterate the same argument using Theorem \ref{Prelim:Lemma:ContinuityProbability_Mixed}.
\end{proof}

\begin{remark}
 \begin{enumerate}[(a)]
  \item Since composition of a white noise with an Euclidean motion as in Remark \ref{Rem:WhiteNoiseModel} again gives a white noise, we can even further improve Theorem \ref{Prelim:Theorem:SpaceTimeWhiteNoiseRegularity_Mixed}. Let for example $B(x_0,r)\subset[0,T]^n$ be a ball in $[0,T]^n$ and consider the restriction of $\tilde{\eta}$ to this ball. By Theorem \ref{Prelim:Theorem:SpaceTimeWhiteNoiseRegularity_Mixed} there is a modification $\eta$ which takes values in $S^{(-\tfrac{1}{2}-\epsilon,\ldots,-\tfrac{1}{2}-\epsilon)}_{p,p}(B(x_0,r)))$ in the Gaussian case. Now we take a rotation $A$ around $x_0$, which is a bijection on $B(x_0,r)$ and an Euclidean motion on $\R^n$. Thus, $\tilde{\eta}\circ A$ is again a white noise so that there is a modification $\eta_1$ which also takes values in $S^{(-\tfrac{1}{2}-\epsilon,\ldots,-\tfrac{1}{2}-\epsilon)}_{p,p}(B(x_0,r)))$. Therefore, for any countable family $(A_n)_{n\in\N}$ of such rotations, there is a modification $\eta$ such that for all $n\in\N$ the rotated noise $\eta\circ A_n$ also takes values in $S^{(-\tfrac{1}{2}-\epsilon,\ldots,-\tfrac{1}{2}-\epsilon)}_{p,p}(B(x_0,r)))$ almost surely. The same argument can also be applied in the non-Gaussian cases.
  \item Theorem \ref{Prelim:Theorem:SpaceTimeWhiteNoiseRegularity} and Theorem \ref{Prelim:Theorem:SpaceTimeWhiteNoiseRegularity_Mixed} are probably not optimal for the general L\'evy case. Looking for example at Theorem \ref{Prelim:Theorem:SpaceTimeWhiteNoiseRegularity_Mixed}, it seems natural to guess that actually $$\mathbb{P}(\eta\in S^{(s,\ldots,s)}_{p,p}([0,T]^n))=1$$ holds.
 \end{enumerate}

\end{remark}

\section{Equations with Boundary Noise}\label{Section:BoundaryNoise}

Now we combine our considerations with the ones on elliptic and parabolic boundary value problems with rough boundary data in \cite{Hummel_2020}. While our results might look quite involved, we would like to point out that there is actually a simple idea behind them: The solutions of the boundary value problems we consider here are arbitrarily smooth. However, as white noise is very rough, there will be singularities at the boundary if one looks at the solution in spaces with higher regularity. The higher the smoothness in time, tangential and normal direction is, the stronger will the singularity be. One can avoid stronger singularities by trading smoothness in the different directions against each other to some extend. The question on how far one can push this will be answered by some technical conditions on the parameters involved. These conditions will make our results look more complicated than they actually are.\\
We should note that Proposition \ref{Prop:Poisson_Equation_Boundary_Noise} on the Poisson equation already follows from the known results, Theorem \ref{Thm:White_Noise_Regularity} and \cite[Theorem 6.1]{Hummel_2020}. Proposition \ref{Prop:Heat_Equation_Boundary_Noise} on the heat equation in turn uses our new result, Theorem \ref{Prelim:Theorem:SpaceTimeWhiteNoiseRegularity}, and \cite[Theorem 6.4]{Hummel_2020}. For this section it is important to keep Example \ref{Example:Weight} in mind. Since $\langle\,\cdot\,\rangle^{\rho}$ is an admissible weight for any $\rho\in\R$, the Besov scale and its dual scale coincide, cf. Proposition \ref{Prop:DualBesov}. Thus, we may apply the results from \cite{Hummel_2020}.\\
There are already several papers in which the singularities at the boundary of solutions of Poisson and heat equation with Dirichlet boundary noise are studied. We refer the reader to \cite{Alos_Bonaccorsi_2002, Brzezniak_et_al_2015, DaPrato_Zabczyk_1993}. This is mainly done by introducing power weights which measure the distance to the boundary of the domain, i.e. weights of the form $\operatorname{\dist}(x,\partial\mathcal{O})^r$ for some $r\in\R$. Such weights are also useful in our approach. But in contrast to \cite{Alos_Bonaccorsi_2002, Brzezniak_et_al_2015, DaPrato_Zabczyk_1993}, we work in spaces with mixed smoothness. This allows us to trade smoothness in normal direction for smoothness in tangential direction. Thus, we can interpret the boundary conditions in a classical sense without having to rely on a mild solution concept.\\
Since we work in $\R^{n}_+$ in this section, the power weight is given by
\[
 |\operatorname{pr}_n|^r\colon \R^n_+\to \R_+, (x_1,\ldots, x_n)\mapsto |x_n|^r.
\]
In this section, we sometimes add subscripts to the domains of function spaces in order to indicate with respect to which variables the spaces should understood. For example we write $B^{s_1}_{p_1,q_1}(\R_t;B^{s_2}_{p_2,q_2}(\R_{+,x_n};B^{s_3}_{p_3,q_3}(\R_{x'}^{n-1})))$ where $\R_t$ corresponds to the time direction, $\R_{+,x_n}$ to the normal direction and $\R_{x'}^{n-1}$ to the tangential directions. $\R^n_{+,x}$ will refer to the space directions.

\begin{proposition}\label{Prop:Poisson_Equation_Boundary_Noise}
 Let $\eta$ be a L\'evy white noise on $\R^{n-1}$ with values in $B^{s}_{p,p}(\R^{n-1}_{x'},\langle\cdot\rangle^{\rho})$ for some parameters $p\in(1,\infty)$, $s,\rho\in\R$, see Theorem \ref{Thm:White_Noise_Regularity}. Let $j\in\{0,1\}$. Then for all $\lambda\in\C\setminus(-\infty,0]$, there is almost surely a unique solution $u\in\mathscr{S}'(\R^n_{+,x})$ of the equation
 \begin{align*}
  \lambda u-\Delta u&=0\quad\text{in }\R^n_{+,x},\\
  \partial_n^j u&=\eta\quad\text{on }\R^{n-1}_{x'},
 \end{align*}
which satisfies
\[
 u\in \bigcap_{r,t\in\R,k\in\N_0,q\in[1,\infty),\atop r-q[t+k-j-s]_+>-1} W^k_q(\R_{+,x_n},|\operatorname{pr}_n|^r;B^{t}_{p,p}(\R^{n-1}_{x'},\langle\cdot\rangle^{\rho})).
\]
Moreover, for all $\sigma>0$, $r,t\in\R$, $k\in\N_0$ and $q\in[1,\infty)$ such that $r-q[t+k-j-s]_+>-1$ there is a constant $C>0$ such that for all $\lambda \in\C\setminus\{\lambda\in\C: |\lambda|>\sigma, |\operatorname{arg} \lambda|<\pi-\sigma\}$ it holds almost surely that
\[
 \|u\|_{W^k_q(\R_{+,x_n},|\operatorname{pr}_n|^r;B^{t}_{p,p}(\R^{n-1}_{x'};\langle\cdot\rangle^{\rho}))}\leq C |\lambda|^{\frac{-1-r+q(k-j)+q[t-s]_+}{2q}} \|\eta\|_{B^{s}_{p,p}(\R^{n-1}_{x'},\langle\cdot\rangle^{\rho})}.
\]
\end{proposition}
\begin{proof}
 This follows directly from combining \ref{Thm:White_Noise_Regularity} and \cite[Theorem 6.1]{Hummel_2020}.
\end{proof}

\begin{remark}
 Note that Proposition~\ref{Prop:Poisson_Equation_Boundary_Noise} yields that $u\in C^{\infty}(\R^n_+)$ with certain singularities which are measured by the weight $|\operatorname{pr}_n|^r$ at the boundary. It is instructive give up some generality in order to see how strong these singularities are in classical function spaces such as $L_2$. Note the L\'evy noises $\eta$ on $\R^{n-1}$ which we consider here always satisfy $\eta\in B^{1+\epsilon-n}_{2,2}(\R^{n-1},\langle\cdot\rangle^{\rho})$ for some $\epsilon>0$. Consider for example the Dirichlet case, i.e. $j=0$. If we take $p=q=2$ and $k=t=0$, then the restriction $r-p[t+k-j-s]_+>-1$ shows that we have to take $r>2n-3-\epsilon$ so that our solution satisfies $u\in L_{2,loc}(\overline{\R^n_+},|\operatorname{pr}_n|^{2n-3})$.
\end{remark}

\begin{proposition}\label{Prop:Heat_Equation_Boundary_Noise}
  Let $\eta$ be a L\'evy white noise on $\R_t\times\R^{n-1}_{x'}$ with values in the space $B^{s_2}_{p_2,\infty,loc}(\R_t;B^{s_1}_{p_1,p_1}(\R^{n-1}_{x'},\langle\cdot\rangle^{\rho}))$ for some parameters $p_1,p_2\in(1,\infty)$, $s_1,s_2,\rho\in\R$, see Theorem \ref{Prelim:Theorem:SpaceTimeWhiteNoiseRegularity}. Let $\tilde{\varphi}\in\mathscr{D}(\R)$, $\varphi=\tilde{\varphi}\otimes \mathbbm{1}_{\R^{n-1}}\in C^{\infty}(\R_t\times\R^{n-1}_{x'})$, $j\in\{0,1\}$ and
  \begin{align*}
   P:=\{(r,t_0,l,k,q):t_0,l\in\R,\, &r\in(-1,\infty),\,k\in\N_0,\,q\in[1,\infty),\\
                        &r-q[t_0+k-j-s]_+>-1,\\
                        &r-2q(l-s_2)-q(k-j)-q[t_0-s_1]_+>-1\}.
  \end{align*}
Then there is almost surely a unique solution $u\in\mathscr{S}'(\R_t\times\R^n_{+,x})$ of the equation
  \begin{align*}
   \partial_tu+u-\Delta u&=0\qquad\;\text{in }\R_t\times \R^n_{+,x'},\\
  \partial_n^j u&=\varphi\cdot \eta\quad\text{on }\R_t\times\R^{n-1}_x,
  \end{align*}
which satisfies
\[
 u\in \bigcap_{(r,t_0,l,k,q)\in P} B^{l}_{p_2,\infty}(\R_t;W^{k}_q(\R_{+,x_n},|\operatorname{pr}_n|^r;B^{t_0}_{p_1,p_1}(\R^{n-1}_{x'},\langle\cdot\rangle^{\rho}))).
\]
Moreover, for all $(r,t_0,l,k,q)\in P$ there is a constant $C>0$ such that almost surely we have the estimate
\[
 \|u\|_{B^{l}_{p_2,\infty}(\R_t;W^{k}_q(\R_{+,x_n},|\operatorname{pr}_n|^r;B^{t_0}_{p_1,p_1}(\R^{n-1}_{x'},\langle\cdot\rangle^{\rho})))}\leq C \|\varphi\cdot \eta \|_{B^{s_2}_{p_2,\infty}(\R_t;B^{s_1}_{p_1,p_1}(\R^{n-1}_{x'},\langle\cdot\rangle^{\rho}))}.
\]
\end{proposition}
\begin{proof}
 This follows directly from combining \ref{Prelim:Theorem:SpaceTimeWhiteNoiseRegularity} and \cite[Theorem 6.4]{Hummel_2020}.
\end{proof}

\begin{remark}
 \begin{enumerate}[(a)]
  \item The reason why we have to multiply $\eta$ with a cutoff function in time is that we only have local results for the regularity in time of a space-time white noise. If there were global results with some weight in time, then we would be able to remove the cutoff function.
  \item As in the elliptic case, we have $u\in C^{\infty}(\R\times\R^n_+)$ with certain singularities at the boundary. This time we have $s_2\geq -1$ and $s_1\geq 1-n$. Thus, if we want to determine a possible weight for the solution of the Dirichlet problem (i.e. $j=0$) to be in a weighted $L_2$-space, we can take $k=t_0=0$, $l>0$ and  $p_2=q=p_1=2$. The restriction $(r,t_0,l,k,q)\in P$ yields that if we take $r>2n+1$, then $u\in L_{2,loc}(\R\times \overline{\R^n_+},|\operatorname{pr}_n|^r)$.
 \end{enumerate}
\end{remark}

\section*{Acknowledgment}
Most of the ideas for this work have been developed during my doctorate and some of them appeared in a similar form in my Ph.D. thesis. Therefore, I thank the Studienstiftung des deutschen Volkes for the scholarship during my doctorate and my supervisor Robert Denk for his great supervision. I also thank the EU for the current support within the TiPES project funded by the European Union's Horizon 2020 research and innovation programme under grant agreement No 820970.

\bibliographystyle{abbrv}

\end{document}